\documentclass[a4paper]{article}

\input{packages.sty}
\input{macros.sty}

\newcommand{\keywords}[1]{} %

\begin{document}
\title{The points of canonical extensions of doctrines}
\author[1]{Sam van Gool}
\author[2]{Joshua L. Wrigley}
\affil[1]{\small LMF, ENS Paris-Saclay}
\affil[2]{\small Department of Mathematics and Statistics, Faculty of Sciences, Masaryk University}
\maketitle

\abstract{We analyse the space of points of the canonical extension of a coherent doctrine. We first give a full characterisation of doctrine morphisms that are extensible, and relate it to the existing notion of p-model of a coherent category. Through this characterisation, the extensible morphisms are shown to be exactly those which are $\omega$-saturated in the sense of coherent first-order logic. 
Next, we answer the question: when does a presheaf of models fully describe the canonical extension? We prove a characterisation theorem via two conditions, which are again natural from the perspective of coherent logic, namely, homogeneity and the realisation of all prime types in a strict sense. The characterisation theorem allows us to deduce a reconstruction result for any coherent theory with the property that all prime types can be realised in a countable, saturated model. For instance, $\omega$-stable coherent theories always have this property. We conclude by explaining how our results can be interpreted topos-theoretically, by relating them to the classifying topos and to the topos of types.
\keywords{Canonical extension $\cdot$ coherent logic $\cdot$ coherent doctrine $\cdot$ saturated model $\cdot$ categorical logic $\cdot$ type space $\cdot$ classifying topos $\cdot$ topos of types}}

\tableofcontents

\section{Introduction}\label{sec:intro}
The duality theory of Stone~\cite{Sto1938} shows that a distributive lattice may be equivalently studied through its space of \emph{points}, which can be viewed as its prime filters, its prime ideals, or its morphisms to the two-element lattice. 
A logical interpretation of Stone duality views a lattice as an algebra of equivalence classes of formulae modulo a propositional theory, and the points of its dual space as model-theoretic types of the theory. 
In the absence of classical negation, types may have non-trivial order relations between them. 
It is a key insight of Priestley's seminal 1970 paper~\cite{Pri1970} that duality theory for distributive lattices is greatly simplified and enhanced if one considers this partial order on points as a primitive, rather than a derived, notion in the space. Under Priestley duality, lattice morphisms  correspond to continuous order-preserving maps between their dual spaces in the opposite direction.

Gehrke and Jónsson~\cite{GehJon1994}, building on work of Jónsson and Tarski for Boolean algebras~\cite{JonTar1951}, introduce an algebraic counterpart to Priestley's dual space: the \emph{canonical extension} $L^\delta$ of a distributive lattice $L$. 
While, importantly, this extension can be constructed in a point-free manner without appeal to any choice principles (cf.~Subsection~\ref{subsec:can-ext-prelims}), a useful spatial point of view on the elements of $L^\delta$ is that they are upwards closed subsets of points of $L$, with respect to Priestley's partial order. 
In other words, the completion $L \into L^\delta$ is an algebraic rendering of the idea of equipping the topological space of points of $L$ with a partial order.

In this way, Priestley duality can,  {\it a posteriori}, be \emph{derived} from the canonical extension construction: 
Given the canonical extension $ L \into L^\delta$ of a distributive lattice $L$,
the set $\Pt(L)$ of points of $L$ is in a natural bijection with the set of \emph{complete} morphisms from $L^\delta$ to the two-element lattice $2$, or, equivalently, with the set $J^\infty L^\delta$ of \emph{completely join-prime} elements of $L^\delta$. 
The Priestley order on $\Pt(L)$ can then be obtained from the restriction of the order of $L^\delta$ to $J^\infty L^\delta$, and the Priestley topology, viewed on $J^\infty L^\delta$, is generated by the sets ${\downarrow}a \cap J^\infty L^\delta$ and their complements, as $a$ ranges over the elements of $L$.

Syntactically speaking, the canonical extension $L^\delta$ may be thought of as an enrichment of the propositional theory that $L$ represents, by adding infinitary conjunction and disjunction, while \emph{only taking into account the finitary part of $L$}. 
This means, for instance, that even if a sequence of elements $(\phi_n)_{n \in \N}$ of $L$ has an infimum $\phi$, say, in $L$, then it will have a new infimum $\phi'$ in $L^\delta$, where $\phi' \not\in L$. This new element $\phi'$ will be the unique successor of $\phi$ in $L^\delta$. In the definition of canonical extensions, this is formalised as a \emph{compactness} property of the completion.
This strict separation of the finitary and the infinitary distinguishes the canonical extension from, for instance, the Dedekind-McNeille completion, and is crucial for the ensuing theory.

In her PhD work~\cite{coumans_phd,coumans_apal} from 2013, Coumans extends the canonical extension construction to \emph{coherent doctrines}, the algebraic structures encoding first-order theories, still in a negation-free setting. Specifically, given a first-order theory $\theory$, one may associate with it a collection $D_\theory$ of distributive lattices, indexed over finite sets of variables called \emph{contexts}, as follows. For each context $c$,  write $D_\theory c$ for the lattice of $\theory$-equivalence classes of formulae of which all free variables are in $c$. Substitutions between contexts then induce lattice morphisms between the fibers. 
This leads to Lawvere's formulation~\cite{lawvere} of \emph{doctrines} as {functors} from a category of contexts to the category of distributive lattices, where the adjective \emph{coherent} signifies that these functors are subject to axioms encoding the rules of (negation-free) first-order logic (called coherent logic by topos theorists; see \cite[\S D1]{elephant}); we recall the precise definitions in Subsection~\ref{subsec:coh-doc-prelims} and discuss syntactic doctrines further in Section~\ref{sec:points}.

The \emph{canonical extension} $D^\delta$ of a coherent doctrine $D$ is  obtained by post-composing $D$ with the functor that assigns to a lattice $L$ its canonical extension $L^\delta$. The first key result of Coumans, recalled in Subsection~\ref{subsec:can-ext-doc}, is that the resulting functor is again a coherent doctrine, i.e., $D^\delta$ still satisfies the axioms encoding first-order logic.
A further deep insight of Coumans's work~\cite{coumans_apal} is that the canonical extension construction, at the first-order level, is closely related to Makkai's \emph{topos of types} construction~\cite{makkai_topos_of_types}, which, as its name suggests, intends to capture the `theory of all model-theoretic types' of a given first-order theory.
This leads us to the question motivating this paper:
\begin{itemize}
	\item[] \emph{How do the type spaces of a first-order theory relate to the canonical extension of its doctrine?}
\end{itemize}

A first answer to this question is essentially immediate from the definitions, generalising the Gehrke-Jónsson viewpoint on Priestley duality outlined above: the types of the theory are still the complete points, in the appropriate sense, of the canonical extension of its doctrine, and it is possible to assemble these into Priestley spaces of types called \emph{polyadic Priestley spaces}, see for instance~\cite{GooMar2024}. However, as already noticed by both Makkai~\cite{makkai_topos_of_types} and Coumans~\cite{coumans_apal}, the appropriate definition of `complete point' for doctrines is more subtle than it is for lattices, and our first contribution here clarifies this subtlety, as we will briefly describe now.

A \emph{point} of a doctrine $D$ is usually conceptualized as a morphism from $D$ to the semantic \emph{doctrine of predicates}, $\mc{P}$, i.e., the contravariant powerset functor. The key problem is that, contrary to the lattice setting, a doctrine $D$ may have points which cannot be extended to complete points of $D^\delta$; we give a concrete example in Example~\ref{exa:not-universal} below.  In Section~\ref{sec:extensible}, we then study the ensuing extensibility problem at the level of general morphisms between doctrines, and we arrive at a characterisation of extensible morphisms in Theorem~\ref{thm:extensible}. Our condition characterising extensible morphisms is similar to one previously identified in Makkai's work~\cite{makkai_topos_of_types} under the name `$p$-model', and this notion also plays a central role in Coumans's work; for more on the comparison, see Remarks~\ref{rem:p-model1}~and~\ref{rem:p-model2}, and Section~\ref{sec:topos} below. Our new contribution in Section~\ref{sec:extensible} is to derive this concept from first principles.

In Section~\ref{sec:points}, we instantiate our general concept of extensibility to points of a doctrine of the form $D_\theory$, for $\theory$ a coherent theory. This allows us to show, in Theorem~\ref{thm:saturated}, that the extensible points of a syntactic doctrine precisely correspond to $\omega$-\emph{saturated} models, a correspondence that had only been implicit in the previous literature (cf.\ Remark~\ref{rem:p-model2}). Here, as everywhere in the paper, the concept of $\omega$-saturation is meant in a negation-free sense that takes into account the partial order between prime types. We discuss the relationship with classical $\omega$-saturation in Remark~\ref{rem:p-model2}.

In Section~\ref{sec:presheaf_pts}, we move from studying individual points of a doctrine to \emph{families} of points. Such families will be indexed by a small category $\mc{K}$, and can be thought of either as presheaves on $\Kcat$ valued in (usual) points of $D$ or as points of $D$ valued in a presheaf category $\Set^\Kcat$; this correspondence is made explicit in Subsection~\ref{subsec:twoways}. The presheaf category $\Set^\Kcat$ induces a semantic doctrine of \emph{subobjects}, generalising the doctrine of predicates $\mc{P}$ mentioned above. 

In Proposition~\ref{prop:sub-psh-dlplus}, we identify the completely join-prime elements of the subobject doctrine of $\Set^\Kcat$ as the \emph{orbits} in a presheaf. Using this characterisation of the completely join-prime elements, we obtain the main result of Section~\ref{sec:presheaf_pts}, Theorem~\ref{thm:iso_iff_types_and_homo}, in which we give two model-theoretic conditions characterising exactly when a presheaf-valued point $M$ `captures' the canonical extension of a doctrine, in the sense that the family induces an isomorphism between $D^\delta$ and the doctrine of subobjects of $\Set^\Kcat$, along a restricted domain. Thus, we recover another important result of Coumans \cite[Theorem 40]{coumans_apal}, as we explain in Section~\ref{sec:topos}. The two model-theoretic conditions on $M$, which we call \emph{homogeneity} and \emph{strictly realising all (prime) types}, characterise the surjectivity and injectivity of the morphism, respectively. This result demonstrates how the canonical extension of doctrines allows for a lattice-theoretic approach to the model theory of first-order logic, in a way that is analogous to the use of canonical extensions of lattices for the model theory of propositional logics, see, e.g.,~\cite{GehPri2007}.

In Section~\ref{sec:reconstr}, we use Theorem~\ref{thm:iso_iff_types_and_homo} to derive an entirely new result: we present a `reconstruction theorem' for the canonical extension of a coherent doctrine of the form $D_\theory$, Theorem~\ref{thm:if_enough_omegasat_then_reconstr}, which applies whenever every prime type of $\theory$ can be realised in a countable, saturated model.  This theorem in particular applies to $\omega$-stable coherent theories, as explained in Corollary~\ref{cor:omega-stable-reconstruction}.  Given some algebraic structure $A$ associated with a theory $\theory$, a `reconstruction problem' asks what further data is needed on the points of $\theory$ in order to be able to reconstruct $A$.  For instance, Priestley duality can be understood as a reconstruction theorem for coherent propositional logic: the Lindenbaum-Tarski algebra of formulae for $\theory$ can be obtained from the points of $\theory$, equipped with a topology.  Our reconstruction results, Theorem~\ref{thm:if_enough_omegasat_then_reconstr} and Corollary~\ref{cor:omega-stable-reconstruction}, assert that, in the particular setting described above, the canonical extension of the syntactic doctrine of $\theory$ can be reconstructed from a certain subcategory of the models of $\theory$ endowed with `knowledge of the underlying sets of the models'.

In the concluding Section~\ref{sec:topos}, we relate our results to the topos-theoretic literature.

\section{Canonical extensions of coherent doctrines}\label{sec:prelims}
In this section, we present a theory of canonical extension for coherent doctrines. This theory was first developed in~\cite{coumans_apal,coumans_phd}, combining doctrine theory~\cite{lawvere,makkai-reyes} and canonical extension theory~\cite{JonTar1951,JonTar1952,GehJon1994,GehJon2004}. We will here recall the definitions necessary for understanding this theory, and prove some preliminary results about them.

\subsection{Coherent doctrines}\label{subsec:coh-doc-prelims}
Distributive lattices\footnote{Throughout this paper, `lattice' means `bounded lattice', i.e., a partially ordered set admitting finite suprema and finite infima, and lattice morphisms are required to preserve finite suprema and finite infima, unless explicitly mentioned otherwise.} provide the algebraic semantics of \emph{coherent propositional logic}.  For the algebraic semantics of \emph{coherent first-order logic}, we replace distributive lattices by \emph{coherent doctrines}~\cite{lawvere}.  The concept of doctrine categorifies the operation which, to any context of variables, associates the distributive lattice of equivalences classes of first-order formulae whose free variables are in the context.  First-order connectives, such as existential quantification and equality, are provided by structure between these lattices, namely the existence of left adjoints.
	\begin{definition}\label{dfn:coh-doc}
		A \emph{coherent doctrine}\footnote{These are also known as `coherent hyperdoctrines' in the literature, but we drop the prefix `hyper' for brevity's sake.} is a functor $P \colon \cat\op \to \DL$ where $\cat$, the \emph{base category}, has all finite limits, and such that, for each arrow $f \colon c \to c'$ in $\cat$, the lattice morphism $Pf \colon Pc' \to Pc$ has a left adjoint, denoted $\exists^P_f \colon Pc \to Pc'$, satisfying:
		\begin{enumerate}
			\item (Frobenius reciprocity) for all $\phi \in Pc$ and $\psi \in Pc'$, 
            \[ \exists^P_f( \phi \land (Pf)(\psi) ) = \exists^P_f(\phi) \land \psi , \text{ and }\] 
			\item (Beck-Chevalley condition) for any pullback square in $\cat$ as given on the left, the square in $\DL$ on the right commutes:
			\[
			\begin{tikzcd}
				c_1 \times_c c_2 \ar{r}{\pi_2} \ar{d}[']{\pi_1} & c_2 \ar{d}{f_2} \\
				c_1 \ar{r}[']{f_1} & c
			\end{tikzcd}
			\qquad
			\begin{tikzcd}
				Pc_2 \ar{r}{P\pi_2} \ar{d}[']{\exists^P_{f_2}} & P c_1 \times_c c_2 \ar{d}{\exists^P_{\pi_1}}  \\
				Pc \ar{r}[']{Pf_1 } & Pc_1 .
			\end{tikzcd}
			\]
		\end{enumerate}
        Let $P \colon \mC\op \to \DL$ and $Q \colon \mD\op \to \DL$ be coherent doctrines. A \emph{morphism} from $P$ to $Q$ is a pair $(M,\alpha)$, where $M \colon \mC \to \mD$ is a finite-limit-preserving functor, and $\alpha \colon P \To Q \circ M\op$ is a natural transformation such that, for any arrow $f \colon c \to c'$ in $\mC$, the following \emph{Beck-Chevalley square} commutes:
        \begin{equation}\label{eq:doc-morphism}
        \begin{tikzcd}
            Pc \ar{r}{\exists^P_f} \ar{d}[']{\alpha_c} & Pc' \ar{d}{\alpha_{c'}} \\
            QMc \ar{r}[']{\exists^Q_{Mf}} & QMc'.
        \end{tikzcd}
        \end{equation}
        Given morphisms $(M,\alpha),(N,\beta) \colon P \to Q$, a \emph{transformation} from $(M,\alpha)$ to $(N,\beta)$ is a natural transformation $\sigma \colon M \Rightarrow N$ such that $\alpha_c(\phi) \leq Q ( \sigma_c) \beta_c (\phi)$ for all $c \in \cat$ and $\phi \in Pc$.  We denote the 2-category of coherent doctrines by $\CohDoc$.
    \end{definition}
    \begin{notation}
        When it can be inferred from the context what is the coherent doctrine $P \colon \cat\op \to \DL$ under consideration, and $f$ is a morphism of $\cat$, we write $f^*$ for the morphism $Pf$ and $\exists_f$ for its left adjoint $\exists^P_f$.
    \end{notation}
    \begin{remark}\label{rem:mono-BC}
        If $P \colon \mC\op \to \DL$ is a coherent doctrine, and $f \colon c \to c'$ is a monomorphism in $\mC$, then the square
        \[
        \begin{tikzcd}
            c \ar{r}{\id_c} \ar{d}[']{\id_c} & c \ar{d}{f}\\
            c \ar{r}[']{f} & c'
        \end{tikzcd}
        \]
        is a pullback in $\mC$. Thus, by the Beck-Chevalley condition, $Pf \circ \exists^P_f = \exists^P_{\id_c} \circ P\id_c = \id_{Pc}$. %
    \end{remark}
	\begin{example}\label{exa:subdoc}
		Let $\cat$ be a category with finite limits which is \emph{well-powered}, i.e., for each object $c$ of $\cat$, the collection $\Sub_\cat(c)$ of subobjects of $c$ is a set. We denote by $\Sub_\cat \colon \cat\op \to \Poset$ the functor that sends each object to its meet-semilattice of subobjects, and acts by pullback on morphisms. The functor $\Sub_\cat$ is a coherent doctrine precisely when $\cat$ is a \emph{coherent category} (see, e.g., \cite[\S A.1]{elephant}). When $\cat = \Set$, the coherent doctrine $\Sub_\Set$ is (isomorphic to) the power set doctrine $\mc{P}$, see Definition~\ref{dfn:power-doc} below. In Section~\ref{sec:presheaf_pts}, we will consider $\Sub_\cat$ for $\cat$ a category of presheaves.
	\end{example}
	\begin{remark}\label{rem:nat-iso-is-doc-morph}
		Let $P \colon \mc{C}\op \to \DL$ and $Q \colon \mc{D}\op \to \DL$ be coherent doctrines.  If $M \colon \mc{C} \to \mc{D}$  is a finite-limit-preserving functor and $\alpha \colon P \Rightarrow Q \circ M\op$ is a natural \emph{isomorphism}, then $(M,\alpha)$ is always a morphism of coherent doctrines $(M,\alpha) \colon P \to Q$.  Indeed, we only need to check the commutativity of \cref{eq:doc-morphism}.  The inverse $\alpha^{-1}_c$ defines a left adjoint to the component $\alpha_c$, and so by taking left adjoints of the commuting naturality square
		\[
		\begin{tikzcd}
			Pc' \ar{r}{Pf}  \ar{d}[']{\alpha_{c'}} & Pc \ar{d}{\alpha_c} \\
			QMc' \ar{r}[']{QMf} & QMc,
		\end{tikzcd}
		\]
		we obtain the equation $\exists^P_f \circ \alpha^{-1}_c = \alpha_{c'}^{-1} \circ \exists^Q_{Mf}$, from which we derive $\alpha_{c'} \circ \exists^P_f = \exists^Q_{Mf} \circ \alpha_c$ as desired.
	\end{remark}
\subsection{Canonical extensions}\label{subsec:can-ext-prelims}
The canonical extension of a lattice~\cite{JT, GehJon1994, GehHar2001, GehJon2004, FusGeh2025} gives an algebraic, point-free presentation of Stone duality and its extensions, without any choice assumptions. We briefly recall here the definition and essential facts about the canonical extension.

Let $L$ be a lattice. A \emph{completion} of $L$ is an injective lattice morphism $\eta \colon L \hookrightarrow C$ where $C$ is a complete lattice. We will always identify $L$ with its isomorphic image $\eta[L]$ in $C$, so that we may consider any completion as a sublattice inclusion $\eta[L] \into C$. In this setting, an element $x \in C$ is called a \emph{filter element} when $x = \bigwedge \{a \in L  \ | \ x \leq a\}$, and $y \in C$ is called an \emph{ideal element} when $y = \bigvee \{b \in L \ | \ b \leq y\}$. The sets of filter and ideal elements of $C$ are denoted $\mc{F}(C)$ and $\mc{I}(C)$, respectively.
	
\begin{definition}
	A completion  $L  \into L^\delta$ is a \emph{canonical extension} of $L$ if it satisfies the following two properties:
	\begin{enumerate}
		\item (density) the filter elements join-generate and the ideal elements meet-generate, i.e., for every $u \in L^\delta$, 
		\[ \bigvee \{ x \in \mc{F}(L^\delta) \ | \ x \leq u \} = u = \bigwedge \{ x \in \mc{I}(L^\delta) \ | \ u \leq y\}, \]
		\item (compactness) for any subsets $S, T$ of $L$, if $\bigwedge S \leq \bigvee T$, then there exist finite subsets $S'$ of $S$ and $T'$ of $T$ such that $\bigwedge S' \leq \bigvee T'$.
	\end{enumerate}
\end{definition}

Every lattice $L$ has a canonical extension, which is unique up to an isomorphism that preserves $L$. This can be proved without appeal to any choice principle~\cite[Section~2]{GehHar2001}.
One such construction exhibits $L^\delta$ as the MacNeille completion of an \emph{intermediate poset}, built from the free directed join completion and the free directed meet completion of $L$; the original source for this abstract viewpoint on the construction is the paper~\cite{GehPri2008} by Gehrke and Priestley. In more concrete terms, as was done in~\cite{GehHar2001}, $L^\delta$ can be built as the complete lattice of stable elements of a Galois connection induced by a relation between the filter lattice of $L$ and the ideal lattice of $L$. 

Let $L$ be a distributive lattice. The canonical extension $L^\delta$ of $L$ is a complete distributive lattice which, assuming a choice principle, can moreover be shown to be isomorphic to the lattice of up-sets of a poset.\footnote{For proofs of the claims in this paragraph and the next, see, e.g.,~\cite[Corollary~3.43, Corollary~3.48 and Theorem~7.5]{FusGeh2025}; note that \cite{FusGeh2025} represents $L^\delta$ using down-sets, but we here use up-sets in order to fit with topos-theoretic tradition.} This (up to isomorphism unique) poset representing $L^\delta$ can be obtained by equipping the set $J^\infty L^\delta$ of \emph{completely join-prime} elements of $L^\delta$ with the opposite of the order of $L^\delta$, restricted to $J^\infty L^\delta$. Here, an element $x$ in a complete lattice $C$ is \emph{completely join-prime} if, for any $S \subseteq C$, $x \leq \bigvee S$ implies that $x \leq s$ for some $s \in S$. In particular, assuming a choice principle, the completely join-prime elements of $L^\delta$ join-generate $L^\delta$. For reasons to be explained further in Section~\ref{sec:points}, we will sometimes refer to completely join-prime elements of $L^\delta$ as \emph{complete points}. 

The poset of completely join-prime elements of $L^\delta$ is isomorphic to the poset of prime filters of $L$ under inclusion, via the function which sends $x \in J^\infty L^\delta$ to $\{a \in L \ | \ x \leq a\}$. 
In summary, we obtain a different construction of $L^\delta$ as the complete lattice of up-sets of the poset underlying the Priestley dual space of $L$.

Recall that a lattice $C$ is isomorphic to a lattice of up-sets of a poset if and only if $C$ is complete and join-generated by its completely join-prime elements (and then, in particular, $C$ is completely distributive); we will call such a lattice a $DL^+$. 
\begin{definition}
We write $\DLplus$ for the category whose objects are $DL^+$s and whose morphisms are \emph{complete morphisms}, i.e., functions that preserve arbitrary joins and arbitrary meets.
\end{definition} 
Note that any $DL^+$ is in particular a complete Heyting algebra, but the morphisms differ; also see Remark~\ref{rem:DL-plus-Heyting}.
		
We will now recall how the object assignment $L \mapsto L^\delta$ extends to a \emph{functor} from the category $\DL$, of distributive lattices with lattice morphisms, to $\DLplus$. First, for any monotone function $f \colon L \to M$, we define, for any $x \in \mc{F}(L)$ and $y \in \mc{I}(L)$,
\[ \bar{f}(x) \isdef \bigwedge \{ f(a) \ | \ a \in L, x \leq a \} \text{ and }
\bar{f}(y) \isdef \bigvee \{ f(b) \ | \ b \in L, b \leq y \}, \]
where we note that $\mc{F}(L) \cap \mc{I}(L) = L$, and the right-hand-sides of both formulae are equal to the value of $f$ when $x$ or $y$ belongs to $L$.
We then define, for any $u \in L^\delta$,
\[ f^\sigma(u) \isdef \bigvee \{ \bar{f}(x) \ | \ x \in \mc{F}(L), x \leq u \} \text{ and } f^\pi(u) \isdef \bigwedge \{ \bar{f}(y) \ | \ y \in \mc{I}(L), u \leq y\}. \]
The extensions $f^\sigma$ and $f^\pi$ are different in general, and neither operation respects composition; see, e.g., \cite[Sections~2.4 and 2.5]{GehJon2004}. However, when $f$ is either finite-join-preserving or finite-meet-preserving, then $f^\sigma = f^\pi$~\cite[Corollary~2.25]{GehJon2004}, and in this case we write $f^\delta$ for the extension. In particular, for any lattice morphism $f \colon L \to M$ between distributive lattices, we have a unique extension $f^\delta \colon L^\delta \to M^\delta$. The following result, Proposition~\ref{prop:can-ext-left-adjoint-functor}, is well-known in the theory of canonical extensions; see e.g. \cite[Theorem~3.2]{GehJon2004},~\cite[Section~1.1]{GehVos2011}, and~\cite[Theorem~8.21 and Corollary~8.22]{FusGeh2025}. Note that distributivity is crucial for the adjointness statement in Proposition~\ref{prop:can-ext-left-adjoint-functor}; also see~\cite{MorAlt2018}.
\begin{proposition}\label{prop:can-ext-left-adjoint-functor}
The assignments $L \mapsto L^\delta$ and $f \mapsto f^\delta$ define a functor $\DL \to \DLplus$, which is left adjoint to the forgetful functor from $\DLplus$ to $\DL$.
\end{proposition}
Proposition~\ref{prop:can-ext-left-adjoint-functor} expresses the universal property that, for every morphism of distributive lattices $f \colon L \to C$, with $C$ a $DL^+$, there exists a unique complete morphism $f^\sharp \colon L^\delta \to C$ such that the following diagram commutes:
\begin{equation} \label{eq:sharp-def}
    \begin{tikzcd}
        & L^\delta \ar[dashed]{d}{f^\sharp} \\
        L \ar[hook]{ru} \ar{r}[']{f} & C.
    \end{tikzcd}
\end{equation}
We call the complete morphism $f^\sharp \colon L^\delta \to C$ the \emph{lifting} of $f$ to $L^\delta$.
\begin{remark}\label{rem:can-ext-finite}
    If $L$ is a \emph{finite} distributive lattice, then $L^\delta = L$, and $f^\sharp = f$ for all $f$.
\end{remark}

Since $f^\sharp$ is a complete morphism, it in particular has a left adjoint. We now prove a technical lemma about this left adjoint which will be important when we consider types in \cref{sec:presheaf_pts}; this lemma is a consequence of known facts in the canonical extensions literature, compare for instance~\cite[Lemma~7.9]{FusGeh2025}. For the reader's convenience, we state and prove it in the form we will use it below.
\begin{lemma}\label{lem:adjoint-type}
	Let $L$ be a distributive lattice, $C$ a $DL^+$, and $\alpha \colon L \to C$ a morphism of distributive lattices. Let $\beta \colon C \to L^\delta$ be the left adjoint of the lifting $\alpha^\sharp \colon L^\delta \to C$. Then, for any completely join-prime element $j \in J^\infty C$, we have 
	\[ \beta(j) = \bigwedge \{ a \in L \ | \ j \leq \alpha(a) \} \ , \]
	where the infimum is taken in $L^\delta$.  Moreover, $\beta(j) \in J^\infty(L^\delta)$.
\end{lemma}
\begin{proof}
The left-to-right inequality is immediate from the facts that $\beta$ is left adjoint to $\alpha^\sharp$, and that $\alpha^\sharp$ extends $\alpha$.
For the right-to-left-inequality, we use that the ideal elements are meet-dense in $L^\delta$.  Let $y \in \mc{I}(L^\delta)$ be arbitrary such that $\beta(j) \leq y$. By the adjunction $\beta \dashv \alpha^\sharp$, we have $j \leq \alpha^\sharp(y)$. By definition of ideal elements and the fact that $\alpha^\sharp$ preserves arbitrary joins, $\alpha^\sharp(y) = \bigvee \{\alpha(a) \ | \ a  \in L, a \leq y\}$. Since $j$ is completely join-prime, pick $a \in L$ such that $a \leq y$ and $j \leq \alpha(a)$. We conclude that $\bigwedge \{a \in L \ | \ j \leq \alpha(a) \} \leq y$.

We now turn to the `moreover' part.  Since $\alpha$ is finite-join-preserving, the set $\{a \in L \ | \ j \leq \alpha(a)\}$ is a prime filter of $L$. Therefore, its infimum is completely join-prime in $L^\delta$ by, e.g.,~\cite[Corollary~3.48]{FusGeh2025}.
\end{proof}
\begin{remark}\label{rem:beta-is-dual}
	Let $f \colon L \to K$ be a morphism between distributive lattices. Post-composing $f$ with the embedding $K \into K^\delta$, we obtain a lattice morphism $f' \colon L \to K^\delta$, and $(f')^\sharp$ coincides with the complete morphism $f^\delta \colon L^\delta \to K^\delta$, by uniqueness in (\ref{eq:sharp-def}). Lemma~\ref{lem:adjoint-type} implies that the left adjoint of $(f')^\sharp = f^\delta$ restricts to a function $g \colon J^\infty K^\delta \to J^\infty L^\delta$. As  explained in Section~\ref{sec:intro}, $J^\infty K^\delta$ and $J^\infty L^\delta$ can be seen as the Priestley dual spaces of $K$ and $L$, respectively, and this function $g$ is then the continuous order-preserving map dual to $f$ under Priestley duality.
\end{remark}
\subsection{Canonical extensions of coherent doctrines}\label{subsec:can-ext-doc}
It is {\it a priori} not obvious that applying the canonical extension construction to a coherent doctrine would again yield a coherent doctrine. That this is true is an important result due to Coumans~\cite[Proposition~9]{coumans_apal}, which we recall in Proposition~\ref{prop:can-ext-of-coh-doc} below.
\begin{definition}
    The \emph{canonical extension} of a coherent doctrine $P \colon \cat\op \to \DL$ is the functor $P^\delta  \colon \cat\op \to \DLplus$, defined as the composite $(-)^\delta \circ P$. We denote by $\eta^P$ the natural transformation $P \To P^\delta$ whose component at an object $c$ of $\cat$ is the inclusion $Pc \into (Pc)^\delta$.
\end{definition}
\begin{proposition}\label{prop:can-ext-of-coh-doc}
    For any coherent doctrine $P \colon \cat\op \to \DL$, the canonical extension $P^\delta \colon \cat\op \to \DLplus$ is a coherent doctrine, and the pair $(\id_\mC, \eta^P)$ is a morphism of coherent doctrines.
\end{proposition}
    Note that, in the statement of Proposition~\ref{prop:can-ext-of-coh-doc}, strictly speaking, $(\id_\mC, \eta^P)$ is a morphism from $P$ to $U \circ P^\delta$, where $U$ is the forgetful functor $\DLplus \to \DL$. When no ambiguity can arise, we will omit such occurrences of $U$ from our notation.
\begin{remark}\label{rem:adjunction-can-ext}
    Adjunctions lift to the canonical extension: if $\exists_f \colon L \to M$ is left adjoint to $f^* \colon M \to L$, then $(\exists_f)^\sigma$ is left adjoint to $(f^*)^\pi$; see, e.g., \cite[Proposition~3.6]{DGP2005}. As a consequence, in the canonical extension $P^\delta$ of a coherent doctrine $P \colon \cat\op \to \DL$, the left adjoint $\exists^{P^\delta}_f$ is the $\sigma$-extension of $\exists^P_f$. When $P$ is clear from the context, we therefore denote this map by $\exists_f^\sigma$.
\end{remark}

\begin{definition}
    A \emph{$DL^+$-doctrine} is a coherent doctrine that takes values in the category $\DLplus$. A \emph{morphism} of $DL^+$-doctrines is a morphism $(M,\alpha)$ of coherent doctrines such that each component of $\alpha$ is moreover a complete morphism. Transformations are as in Definition~\ref{dfn:coh-doc}.  We denote the 2-category of $DL^+$-doctrines by $\CohDocplus$.
\end{definition}
\begin{remark}\label{rem:DL-plus-Heyting}
    As pointed out earlier, any $DL^+$ is in particular a Heyting algebra, but morphisms in the category $\DLplus$ do not necessarily preserve the Heyting implication. However, if $Q \colon \mD\op \to \DLplus$ is a $DL^+$-doctrine, then $Qf$ must preserve the Heyting implication for any arrow $f$ in $\mD$. This follows from a more general fact: if $Q \colon \mD\op \to \DL$ is a functor such that $Qd$ is a complete Heyting algebra for all objects $d$ in $\mD$, then the Frobenius condition is equivalent to the statement that $Qf$ preserves the Heyting implication for all arrows $f$ in $\mD$; see, e.g.,~\cite[Proposition~V.1.1]{JT}.  %
\end{remark}
Another consequence of the Frobenius condition, which we will use in Section~\ref{sec:extensible}, is the following.
\begin{lemma}\label{lem:non-empty-meets}
    Let $e \colon L \leftrightarrows L' : h$ be an adjunction between complete lattices such that $h \circ e = \id_L$, and suppose that, for all $a \in L$ and $b \in L'$, 
   the Frobenius condition $e(a) \wedge b = e(a \wedge h(b))$ is satisfied. Then $e$ preserves all non-empty meets.
\end{lemma}
\begin{proof}
    Let $S \subseteq L$ be non-empty. Write $s_0 := \bigwedge S$ and $s_0' := \bigwedge e[S]$. Fix $s \in S$. Since $s_0' \leq e(s)$, we have 
    \[ s_0' = e(s) \wedge s_0' = e(s \wedge h(s_0')) \ . \]
    Using that $h$ preserves all meets, and that $h \circ e = \id_L$, we have 
    \[ h(s_0') = \bigwedge_{s \in S} h(e(s)) = s_0 \ . \] 
    Combining the two equalities, we see that $s_0' = e(s \wedge s_0) = e(s_0)$, as required.
\end{proof}

For any category $\mc{C}$, the adjunction $\DL \leftrightarrows \DLplus$ yields an adjunction between functor categories $[\mc{C}\op,\DL] \leftrightarrows [\mc{C}\op,\DLplus]$, where the unit and counit of the latter adjunction are given pointwise by the unit and counit of $\DL \leftrightarrows \DLplus$.  In particular, we have the following result.
\begin{proposition}\label{prop:sharp-natural}
    Let $P \colon \mC\op \to \DL$ be a coherent doctrine, let $Q \colon \mD\op \to \DLplus$ be a $DL^+$-doctrine, and let $(M,\alpha) \colon P \to Q$ be a coherent morphism. There exists a unique natural transformation $\alpha^\sharp \colon P^{\delta} \To Q \circ M\op$ such that, for each object $c$ of $\mC$, the restriction of $\alpha^\sharp_c$ to $Pc$ is equal to $\alpha_c$.
\end{proposition}
By Proposition~\ref{prop:can-ext-of-coh-doc}, the object assignment $P \mapsto P^\delta$ associates to any coherent doctrine a $DL^+$-doctrine, and this assignment readily extends to a functor from $\CohDoc$ to $\CohDocplus$. However, as already observed in the work of Coumans~\cite[p.~150]{coumans_phd}, this assignment does \emph{not} satisfy a universal property analogous to (\ref{eq:sharp-def}), i.e., the functor $(-)^\delta$ is not left adjoint to the forgetful functor from $\CohDocplus$ to $\CohDoc$. %

Let $(M,\alpha) \colon P \to Q$ be a morphism of coherent doctrines with $Q$ a $DL^+$-doctrine.  Although, via Proposition~\ref{prop:sharp-natural}, we always obtain a unique natural transformation $\alpha^\sharp \colon P^\delta \Rightarrow Q \circ M\op$ whose restriction is $\alpha$, the pair $(M,\alpha^\sharp)$ can fail to be a morphism of $DL^+$-doctrines because the square
\[
\begin{tikzcd}
	P^\delta c \ar{r}{\exists^\sigma_f} \ar{d}[']{\alpha^\sharp_c}& P^\delta c' \ar{d}{\alpha^\sharp_{c'}} \\
	QMc \ar{r}[']{\exists^Q_{Mf}} & QMc'
\end{tikzcd}
\]
may not commute.  We will give an explicit instance of such a failure in Example~\ref{exa:not-universal}, involving the following two doctrines, which will continue to play an important role in the later sections.
\begin{definition}\label{dfn:power-doc}
    The \emph{power set $DL^+$-doctrine} is the functor $\Powplus \colon \Set\op \to \DLplus$ which sends a set $S$ to the complete lattice $\mc{P}S$ of subsets of $S$, ordered by inclusion, and a function $f \colon S \to T$ to the inverse image morphism $\Powplus f := f^{-1} \colon \mc{P}T \to \mc{P}S$.

    The \emph{power set coherent doctrine} is the functor $\Pow := U \circ \Powplus$, where $U \colon \DLplus \to \DL$ is the forgetful functor.
\end{definition}
\begin{remark}
    As the names suggest, $\Powplus$ is a $DL^+$-doctrine and $\Pow$ is a coherent doctrine. When no ambiguity can arise, we sometimes also write $\Pow$ for $\Powplus$ and refer to either as `the power set doctrine'.
\end{remark}
\begin{example}\label{exa:not-universal}
    Proposition~\ref{prop:sharp-natural} gives a natural transformation $u^\sharp \colon \Pow^\delta \Rightarrow \Powplus$, where $u$ denotes the identity transformation $\Pow = U \circ \Powplus$. We will now show that the pair $(\id_{\Set}, u^\sharp)$ is \emph{not} a morphism of coherent doctrines, because not all the squares for the left adjoints in the square (\ref{eq:doc-morphism}) commute. Denote by $!$ the unique function from $\mbb{N}$ to $1$. We claim that the following square does not commute: 
    \[
        \begin{tikzcd}
            \mc{P}(\mbb{N})^\delta \ar{r}{\exists^{\mc P^\delta}_!} \ar{d}[']{u_\mbb{N}^\sharp} & \mc{P}(1)^\delta \ar{d}{u_1^\sharp} \\
            \mc{P}(\mbb{N}) \ar{r}[']{\exists^{\mc P}_{!}} & \mc{P}(1).
        \end{tikzcd}
    \]
    Let $\mathfrak{F}$ be the collection of co-finite subsets of $\mbb N$, and write $x := \bigwedge \mathfrak{F}$ for the infimum of $\mathfrak{F}$, computed in $\mc{P}(\mbb N)^\delta$.
    Now, using that $u_\mbb{N}^\sharp$ is a complete morphism extending the identity on $\mc{P}(\mbb N)$, we have
    \[ u_{\mbb N}^\sharp(x) = \bigcap_{a \in \mathfrak{F}} u_{\mbb N}^\sharp(a) = \bigcap \mathfrak{F} = \emptyset \ . \] 
    Thus, $\exists_!^\mc{P} u_{\mbb{N}}^\sharp(x) = \emptyset$, since $\exists_!^\mc{P}$, being a left adjoint, preserves the minimum.

	On the other hand, we will now prove that $u_1^\sharp(\exists_!^{\mc{P}^\delta}x) \neq \emptyset$. By Remark~\ref{rem:can-ext-finite}, $u_1^\sharp = u_1$, the identity on $\mc{P}(1)$, so it is equivalent to show that $\exists_!^{\mc{P}^\delta}x \neq \emptyset$. Since $\emptyset = \mc{P}^\delta!(\emptyset)$, applying the fact that $\exists_!^{\mc{P}^\delta}$ is left adjoint to $\mc{P}^\delta!$, we see that this is in turn equivalent to $x \neq \emptyset$ in $\mc{P}(\mbb{N})^\delta$. The latter holds by compactness, since any finite subset of $\mathfrak{F}$ has a non-empty intersection.
	Thus, $u^\sharp_1 \circ \exists^{\power^\delta}_!  \neq \exists_!^\power \circ u^\sharp_\N$ as desired.
\end{example}
Example~\ref{exa:not-universal} prompts us to pose the question: for \emph{which} morphisms from a coherent doctrine to a $DL^+$-doctrine is the lifting still a morphism? We will provide an answer in Section~\ref{sec:extensible}. 

\section{Extensible morphisms of coherent doctrines}\label{sec:extensible}
The main result of this section is the following characterisation of morphisms from a coherent doctrine to a $DL^+$-doctrine that admit liftings; we will call such morphisms \emph{extensible} (see Definition~\ref{dfn:extensible} below).
\begin{theorem}\label{thm:extensible}
    Let $P \colon \mC\op \to \DL$ be a coherent doctrine, let $Q \colon \mD\op \to \DLplus$ be a $DL^+$-doctrine, and let $(M,\alpha) \colon P \to Q$ be a coherent morphism. The following are equivalent:
    \begin{enumerate}
        \item \label{it:restrict} there exists a morphism $P^\delta \to Q$ in $\CohDocplus$ whose restriction to $P$ is $(M,\alpha)$;
        \item the pair $(M,\alpha^\sharp) \colon P^\delta \to Q$ is a morphism in $\CohDocplus$;
        \item \label{it:proj-lax} for any objects $c_1, c_2$ of $\mC$, and $\pi_1 \colon c_1 \times c_2 \to c_1$ the left projection,
the following diagram lax commutes:
    \[
        \begin{tikzcd}
   P^\delta(c_1 \times c_2) \ar{d}[']{\alpha_{c_1\times c_2}^\sharp} \ar{r}{(\exists_{\pi_1}^P)^\sigma} & P^\delta c_1 \ar{d}{\alpha_{c_1}^\sharp} \arrow[ld, "\geq"{sloped, description}, phantom] \\ 
   QM(c_1\times c_2) \ar{r}[']{\exists_{M\pi_1}^Q}        & QMc_1.
        \end{tikzcd}
    \] 
    \end{enumerate}
\end{theorem}
In \ref{it:restrict} of Theorem~\ref{thm:extensible}, by the \emph{restriction} of a morphism $(N,\beta) \colon P^\delta \to Q$ in $\CohDocplus$, we mean the pair $(N,\beta') \colon P \to Q$, where the component of $\beta'$ at an object $c$ of $\mC$ is the restriction of the complete morphism $\beta_c \colon P^\delta c \to QNc$ to the sublattice $Pc \subseteq P^\delta c$.

In the remainder of this section, let $P \colon \mC\op \to \DL$ be a coherent doctrine, let $Q \colon \mD\op \to \DLplus$ be a $DL^+$-doctrine, and let $(M,\alpha) \colon P \to Q$ be a coherent morphism. 
The crucial step towards proving Theorem~\ref{thm:extensible} is the following lemma.
\begin{lemma}\label{lem:sharp-automatic}
    For any arrow $f \colon c \to c'$ in $\mC$, the following diagram lax commutes:
    \[
        \begin{tikzcd}
   P^\delta c\ar{d}[']{\alpha_{c}^\sharp} \ar{r}{\exists_{f}^\sigma} & P^\delta c' \ar{d}{\alpha_{c'}^\sharp} \arrow[ld, "\leq"{sloped, description}, phantom] \\ 
   QM c \ar{r}[']{\exists_{Mf}}        & QMc'.
        \end{tikzcd}
    \] 
    If, moreover, $f$ is a monomorphism, then the diagram strictly commutes.
\end{lemma}
\begin{proof}
   Essentially, the first statement of the lemma is appealing to the categorical \emph{mate}.  Explicitly, since $\exists_f^\sigma$ is left adjoint to $P^\delta f$ by Remark~\ref{rem:adjunction-can-ext}, we have $\id_{P^\delta c} \leq P^\delta f \circ \exists_f^\sigma$. Therefore, $\alpha^\sharp_c \leq \alpha^\sharp_c \circ P^\delta f \circ \exists_f^\sigma$. By naturality of $\alpha^\sharp$ (Proposition~\ref{prop:sharp-natural}), $\alpha^\sharp_c \circ P^\delta f = QMf \circ \alpha^\sharp_{c'}$, so $\alpha^\sharp_c \leq QMf \circ \alpha^\sharp_{c'} \circ \exists_f^{\sigma}$. Since $\exists_{Mf}$ is left adjoint to $QM f$, the desired inequality $\exists_{Mf} \circ \alpha_c^\sharp \leq \alpha^{\sharp}_{c'} \circ \exists^\sigma_f$ now follows.

    If $f$ is a monomorphism, then so is $Mf$, since $M$ preserves finite limits. Thus, by Remark~\ref{rem:mono-BC}, we obtain $QMf \circ \exists_{Mf} = \id_{Mc}$.  By Lemma~\ref{lem:non-empty-meets}, $\exists_{Mf}$ thus preserves all non-empty meets. Now note that it suffices to prove the equality $\alpha^\sharp_{c'}(\exists_f^\sigma x) = \exists_{Mf}(\alpha_c^\sharp x)$ for a filter element $x \in \mc F(Pc)$, since any element of $P^\delta c$ is a join of such elements, and both $\alpha^\sharp_{c'} \circ \exists_f^\sigma$ and $\exists_{Mf} \circ \alpha_{c}^\sharp$ preserve arbitrary joins.

    Let $x \in \mc F(Pc)$ be arbitrary. We have $x = \bigwedge \{ a \in Pc, x \leq_{P^\delta c} a \}$, and this meet is non-empty since $\top_{Pc}$ is in the set on the right hand side. Now
\begin{align*}
    \alpha_{c'}^\sharp(\exists_f^\sigma x) &= \alpha_{c'}^\sharp\left(\bigwedge \{ \exists_f a \ | \ a \in Pc, x \leq a\}\right) & \text{(definition of $\exists_f^\sigma$)}\\ 
    &= \bigwedge \{\alpha_{c'} (\exists_f a) \ | \ a \in Pc, x \leq a\} & \text{($\alpha_{c'}^\sharp$ preserves meets and lifts $\alpha_{c'}$)}\\ 
    &=  \bigwedge \{\exists_{Mf} (\alpha_c a) \ | \ a \in Pc, x \leq a\} & \text{($(M,\alpha)$ is a morphism of doctrines)} \\ 
    &= \exists_{Mf} (\alpha_c^\sharp x). & \text{($\exists_{Mf} \circ \alpha_{c}^\sharp$ preserves non-empty meets), }
\end{align*}
which is the required equality.
\end{proof}

\begin{proof}[Proof of Theorem~\ref{thm:extensible}]
    (i) $\iff$ (ii). If there exists a morphism of $DL^+$-doctrines $(M, \beta) \colon P^\delta \to Q \circ M\op$ whose restriction to $P$ is $(M,\alpha)$, then $\beta$ must be natural, and thus equal to $\alpha^\sharp$ by Proposition~\ref{prop:sharp-natural}. Conversely, if $(M,\alpha^\sharp)$ is a morphism, then its restriction is clearly $(M,\alpha)$.

    (ii) $\iff$ (iii). Since $\alpha^\sharp$ is always a natural transformation from $P^\delta$ to $Q\circ M\op$ (Proposition~\ref{prop:sharp-natural}), (ii) is equivalent to the condition that all squares as in (\ref{eq:doc-morphism}) commute for $\alpha^\sharp$. By Lemma~\ref{lem:sharp-automatic}, this is in turn equivalent to the statement that, for every $f \colon c \to c'$, the following square lax commutes:
    \begin{equation}\label{eq:lax-goal}
        \begin{tikzcd}
            P^\delta c \ar{r}{\exists^{\sigma}_f} \ar{d}[']{\alpha^\sharp_c} & P^\delta c' \ar{d}{\alpha^\sharp_{c'}}  \arrow[ld, "\geq"{sloped, description}, phantom]\\
            QMc \ar{r}[']{\exists_{Mf}} & QMc'.
        \end{tikzcd}
    \end{equation}
    It is thus clear that (ii) $\Rightarrow$ (iii). Assume (iii) holds, and let $f \colon c \to c'$ be an arbitrary arrow in $\mC$. Since $\mC$ has products, we can take the \emph{graph} $G_f$ of $f$, i.e.\ the universally induced arrow
		\[
		\begin{tikzcd}
			& c \ar{ld}[']{f} \ar[dashed]{d}{G_f} \ar[equals]{rd} \\
			c' & \ar{l}{\pi_1} c' \times c \ar{r} & c.
		\end{tikzcd}
		\]
        The graph of any arrow is a monomorphism, so the equality $\alpha^\sharp_{c' \times c} \circ \exists_{G_f}^\sigma = \exists_{M(G_f)} \circ \alpha^\sharp_{c}$ holds by Lemma~\ref{lem:sharp-automatic}. Moreover, the inequality $\alpha_{c'}^\sharp \circ \exists_{\pi_1}^\sigma \leq \exists_{M \pi_1} \circ \alpha^\sharp_{c' \times c}$ holds by assumption.
        Therefore, using that left adjoints are functorial, we obtain
        \[ \alpha_{c'}^\sharp \circ \exists_f^\sigma = \alpha_{c'}^\sharp \circ \exists_{\pi_1}^\sigma \circ \exists_{G_f}^\sigma \leq \exists_{M \pi_1} \circ \alpha^\sharp_{c' \times c} \circ \exists_{G_f}^\sigma = \exists_{M\pi_1} \circ \exists_{M(G_f)} \circ \alpha^\sharp_c = \exists_{Mf} \circ \alpha^\sharp_c, \]
        which is the required inequality. 
\end{proof}

\begin{definition}\label{dfn:extensible}
    We say that a morphism $(M,\alpha) \colon P \to Q$, with $P$ a coherent doctrine and $Q$ a $DL^+$-doctrine, is \emph{extensible} if it satisfies the equivalent properties of Theorem~\ref{thm:extensible}.
\end{definition}

We record an immediate consequence of Theorem~\ref{thm:extensible} for future reference.
\begin{corollary}\label{cor:extensible-cjp}
    Let  $P \colon \mC\op \to \DL$ be a coherent doctrine, $Q \colon \mD\op \to \DLplus$ a $DL^+$-doctrine, and $(M, \alpha) \colon P \to Q$ a morphism. Suppose that, for each object $c$ of $\mC$, $X(c)$ is a subset of $P^\delta(c)$ which join-generates $P^\delta(c)$. Then $(M,\alpha)$ is extensible if and only if for any objects $c_1,c_2$ of $\mc C$ and any $x \in X(c_1 \times c_2)$,
    \begin{equation}\label{eq:extensible-inequality}
    \alpha_{c_1}^\sharp(\exists_{\pi_1}^\sigma(x)) \leq \exists_{M\pi_1}(\alpha^\sharp_{c_1\times c_2}(x)) .
    \end{equation}
\end{corollary}
\begin{proof}
    Since all maps involved preserve arbitrary joins, the stated condition is equivalent to \cref{it:proj-lax} in Theorem~\ref{thm:extensible}.
\end{proof}
Note that Corollary~\ref{cor:extensible-cjp} applies when $X(c) \subseteq P^\delta(c)$ is taken to be the subset of filter elements $\mc{F}(Pc)$, or, assuming a choice principle, the subset of of completely join-prime elements $J^\infty (P^\delta c)$. 

\begin{remark}\label{rem:p-model1}
	We explain how our notion of extensibility is a doctrinal version of the notion of $p$-model that figures in \cite{makkai_topos_of_types} and \cite{coumans_apal}.
In light of the equivalence between completely join-prime elements and prime filters and of the definitions of $(-)^\sigma$ and $(-)^\sharp$ given in Section~\ref{sec:prelims}, using Corollary~\ref{cor:extensible-cjp}, we see that $(M,\alpha)$ is extensible if and only if, for  any objects $c_1, c_2$ of $\cat$ and any prime filter $\rho$ in $P(c_1\times c_2)$, we have the inequality
\[ \bigwedge_{a \in \rho} \alpha_{c_1}\exists_{\pi_1}(a) \leq \exists_{M\pi_1}\left(\bigwedge_{a \in \rho} \alpha_{c_1 \times c_2}(a)  \right) \ . \] 
In~\cite[Definition~5.2.7]{coumans_phd}, following \cite[Section~1]{makkai_topos_of_types}, a $p$-model is defined as a coherent functor $F \colon \mC \to \mD$ between coherent categories such that, for any morphism $f \colon c \to c'$ in $\mC$ and any prime filter $\rho$ in $\Sub_{\mC}(c)$, we have the equality
\[ \bigwedge_{a \in \rho} \exists_{Ff}(Fa) = \exists_{Ff}\left( \bigwedge_{a \in \rho} Fa\right)\ .\] 
As also noted there, this equality is equivalent to the left-to-right inequality, as the other inequality always holds.
Since $F$ is coherent, the left-hand-side can be rewritten as $\bigwedge_{a \in \rho} F\exists_f (a)$. 
From this and our work earlier in this section, we can conclude that our notion of extensibility is a doctrinal version of the definition of $p$-model of \cite{makkai_topos_of_types} and \cite{coumans_apal}.
\end{remark}
\begin{example}
	Let $\DL_f \subseteq \DL$ denote the subcategory of finite distributive lattices, and let $\Poset_f$ denote the category whose objects are finite posets and whose morphisms are monotone functions.  Consider the functor $[-,\two] \colon \Poset_f\op \to \DL_f \subseteq \DL$, where a finite poset $P$ is sent to the distributive lattice of monotone functions from $P$ into $\two$ equipped with the pointwise ordering (equivalently, the lattice of up-sets of $P$; recall that this is the functor that witnesses Birkhoff duality).  This functor describes a coherent doctrine, where the left adjoint to $- \circ f \colon [Q,\two] \to [P,\two]$, for a monotone map $f \colon P \to Q$, sends $g \in [P,\two]$ to the map $\exists_{f}(g) \colon Q \to \two$ given by 
	\begin{align*}
	\exists_{f}g (q) & = \bigvee \set{g(p)}{p \in P \text{ such that} f(p) \leq q} , \\
	& = 
	\begin{cases}
		\top & \text{if there is $p \in P$ with $f(p) \leq q$ and $g(p) = \top$,} \\
		\bot & \text{otherwise.}
	\end{cases}
	\end{align*}
	In an analogous fashion, for any distributive lattice $A$, the functor $[-,A] \colon \Poset_f\op \to \DL$ is a coherent doctrine.
	
	Let us consider morphisms of the form $(\id_{\Poset_f},\alpha) \colon [-,A] \to [-,\two]$.  %
	Note that $[-,\two]$ is a $DL^+$-doctrine (indeed, each fibre is finite), and so it makes sense to ask what it means for $(\id_{\Poset_f},\alpha)$ to be extensible.  By Corollary~\ref{cor:extensible-cjp}, it suffices to check that, for any pair of finite posets $P, Q$ and a filter element $x \in \mc{F}([P \times Q,A])$, we have that $\alpha^\sharp_P(\exists_{\pi_1}^\sigma(x)) \leq \exists_{\pi_1}(\alpha^\sharp_{P \times Q}(x))$, i.e.\ there is an inclusion of the corresponding up-sets.  Recall, from the proof of Lemma~\ref{lem:sharp-automatic}, that 
	\[
	\alpha^\sharp_P(\exists_{\pi_1}^\sigma(x)) = \bigwedge \lrset{\exists_{\pi_1}\alpha_{P \times Q}(g)}{g \in [P \times Q,A], \, x \leq g}.
	\]
	Thus, an element $p \in P$ is contained in the up-set $\alpha^\sharp_P(\exists_{\pi_1}^\sigma(x)) \subseteq P$ if and only if, for all $g \in [P \times Q, A]$ with $g \geq x$, there exists some $q_g \in Q$ such that the pair $(p,q_g)$ is contained in the up-set $\alpha_{P \times Q} (g) \subseteq P \times Q$.  Conversely, we readily calculate that $p$ is contained in $\exists_{\pi_1}(\alpha^\sharp_{P \times Q}(x))$ if and only if there exists some $q \in Q$ such that, for all $g \geq x$, the pair $(p,q)$ is contained in $\alpha_{P \times Q} (g)$.  Thus, if we are given a family of $q_g$'s that depend on $g$ for which $(p,q_g) \in \alpha_{P \times Q} (g)$, we must find a single $q' \in Q$ for which $(p,q')  \in \alpha_{P \times Q} (g)$.  In particular, if $Q$ were a join-semilattice, then we can ensure the existence of such a $q'$ by taking the join $q' \isdef \bigvee_{g \geq x} q_g$ since, by monotonicity, we arrive at
	\[
	(p,q_g) \in \alpha_{P \times Q} (g) \implies (p,q') \in \alpha_{P \times Q} (g).
	\]
	This pattern of seeking a uniform witness from a dependent choice prefigures the introduction of model-theoretic saturation in the next section.
\end{example}

\section{Points}\label{sec:points}
In this section, we give a model-theoretic interpretation to the notion of extensible morphism, and in Theorem~\ref{thm:saturated} we establish its close connection with $\omega$-saturation. To this end, we first recall the different notions of \emph{point} of a doctrine, and how they correspond to \emph{models} of a theory.

For $L$ a distributive lattice, we call a morphism from $L$ to the two-element distributive lattice $\two$ a \emph{coherent point} of $L$. A coherent point $h \colon L \to \two$ uniquely determines a prime filter $x_h := h^{-1}(1)$ of $L$. For $C$ a $DL^+$, we call a complete morphism from $C$ to $\two$ a \emph{complete point} of $C$. A complete point $h \colon C \to \two$ uniquely determines a completely join prime element $x_h := \bigwedge h^{-1}(1)$ of $C$.
By the universal property (\ref{eq:sharp-def}), the complete points of $L^\delta$ are in bijection with the coherent points of $L$.

When defining points in the context of coherent and $DL^+$ doctrines, the distributive lattice $\two$ is replaced by the power set doctrines $\Pow$ and $\Powplus$ (see Definition~\ref{dfn:power-doc}), respectively.
\begin{definition}\label{dfn:points}
    Let $P$ be a coherent doctrine. A \emph{coherent point} of $P$ is a morphism $(M,\alpha) \colon P \to \Pow$ in $\CohDoc$.  We denote the category of points of $P$ and their transformations by $\Pt(P)$.

    Let $Q$ be a $DL^+$-doctrine. A \emph{complete point} of $Q$ is a morphism $(M,\alpha) \colon Q \to \Powplus$ in $\CohDocplus$.  We denote the category of complete points of $Q$ and their transformations by $\Ptc(Q)$.
\end{definition}
Let $P$ be a coherent doctrine. Instantiating Definition~\ref{dfn:extensible} with $Q := \Powplus$, we see that the complete points of $P^\delta$ are in bijection with the \emph{extensible} coherent points of $P$.

Let us now briefly recall the connection between coherent points and models of a coherent theory. For simplicity, we restrict ourselves to the case of coherent theories that are single-sorted and relational, i.e., having no function symbols. A \emph{coherent formula} is built from atomic predicates using logical operators $\wedge$, $\vee$, $\bot$,  $\top$ and $\exists$, and a \emph{coherent theory} is then formulated using appropriate sequent derivation rules of the form ${\phi({x}) \vdash \psi({x})}$, which should be read as ``for all ${x}$, if $\phi({x})$ then $\psi({x})$'' (for more on the syntax of coherent logic, see \cite[\S D1]{elephant}). A \emph{model} $\mc M$ of a coherent theory $\theory$ is a set, together with an interpretation of each symbol of the signature of $\theory$, such that all sequents in $\theory$ are interpreted in $\mc M$ as true statements. 

In Definition~\ref{dfn:syntactic-doctrine}, we recall how to build a syntactic doctrine for coherent theory, which is a coherent doctrine on base category $\Finset\op$, and in Remark~\ref{rem:points-as-models} we recall how coherent points of this doctrine correspond to models of the coherent theory. For further details and proofs, 
see, e.g., \cite[\S~8]{makkai-reyes} or \cite[\S~5]{coumans_phd}.
\begin{definition}\label{dfn:syntactic-doctrine}
    Let $\theory$ be a (single-sorted, relational) coherent theory. The \emph{syntactic doctrine} of $\theory$ is the functor $D_\theory \colon \Finset \to \DL$ which sends a finite set $X$ to the distributive lattice of $\theory$-equivalence classes of $\theory$-formulae whose free variables are in $X$, and which sends a function $r \colon X \to Y$ to the homomorphism $D_\theory r \colon D_\theory X \to D_\theory Y$ induced by the substitution function $\phi(x_1,\dots,x_n) \mapsto \phi(r(x_1),\dots,r(x_n))$. 
\end{definition}
In what follows, by a common abuse of notation, we will treat elements of $D_\theory X$ as formulae, rather than as $\theory$-equivalence classes of formulae. This abuse of notation will be harmless since all our arguments can be done modulo $\theory$-equivalence.
\begin{remark}\label{rem:descr_of_left_adjoint_syntactic}
	The functor $D_\theory \colon \Finset \to \DL$ defines a coherent doctrine.  For a function $ r \colon X \to Y$, the left adjoint to $D_\theory r$ is the map $\exists^{D_\theory}_r \colon D_\theory Y \to D_\theory X$ given by 
	\[
	\psi(y_1, \dots , y_m) \mapsto \exists y_1, \dots , y_m . \, \psi(y_1, \dots y_m) \land x_1 = y_{r(1)} \land \dots \land x_n = y_{r(n)}, 
	\]
	where $y_{r(i)}$ denotes the image of $x_i \in X$ under $r$. In particular, for a coproduct inclusion $i_X \colon X \to X + Y$, we have $\exists^{D_\theory}_{i_X}(\psi(x,y)) \equiv \exists y . \, \psi(x,y)$ (where $x$ and $y$ now denote tuples of variables), and for the co-diagonal map $\nabla_X \colon X + X\to X$, we have $\exists^{D_\theory}_{\nabla_X}(\phi(x)) \equiv \phi(x) \land x = x'$.  For further details that this satisfies the Beck-Chevalley condition, etc., see \cite{seely}.
\end{remark}
\begin{remark}\label{rem:points-as-models}
    When $\theory$ is a coherent theory and $D_\theory$ its syntactic doctrine, the coherent points of $D_\theory$ correspond to the models of $\theory$: if $\mc M$ is a model with underlying set $M$, then we can define a morphism $(M, \class{-}^{\mc M}) \colon D_\theory \to \Pow$ as follows. Let $M \colon \Finset\op \to \Set$ be the functor defined on objects by sending a finite set $X$ to the set $M^X$ of $X$-tuples in $M$, and defined on morphisms by sending a function $r \colon X \to Y$ to the function $Mr \colon M^Y \to M^X$, which associates any $Y$-tuple $m \in M^Y$ with the $X$-tuple $m \circ r \in M^X$. For any formula $\phi$ with free variables in $X$, let $\class{\phi}^{\mc M}_X := \{ m \in M^X \ | \ M \models \phi(m)\}$, i.e.\ the interpretation of $\phi$ in $\mc{M}$. Since $\mc{M} \models \theory$ is a model, this assignment is well-defined on $\theory$-equivalence classes of formulae, and the pair $(M, \class{-}^{\mc M})$ thus defined is a morphism in $\CohDoc$. We call $(M, \class{-}^{\mc M})$ the \emph{associated point} of $\mc M$. When $\mc M$ is clear from the context, we will write $\class{-}$ instead of $\class{-}^{\mc M}$. Note that the model $\mc M$ can be recovered from its associated point: the underlying set is the value of $M$ at a singleton set and $\class{-}^{\mc M}$ in particular records the interpretation of all predicate symbols. 
    This assignment $\mc M \mapsto (M, \class{-}^{\mc M})$ is the object part of an isomorphism between the category of models of $\theory$ and the category of points of $D_\theory$ (where the morphisms of points are the transformations of Definition~\ref{dfn:coh-doc}). See, e.g., \cite[\S~1]{Wri2026} for more details.
\end{remark}

In Theorem~\ref{thm:saturated}, we will instantiate Corollary~\ref{cor:extensible-cjp} to the specific case of a point associated to a model of a coherent theory, in the sense of Remark~\ref{rem:points-as-models}. Working towards this, let us calculate what the two sides of the inequality (\ref{eq:extensible-inequality}) in Corollary~\ref{cor:extensible-cjp} look like in this specific case.

Let $\mc M$ be a model of a coherent theory $\theory$, with associated point $(M, \class{-}) \colon D_\theory \to \Pow$. For any finite sets $X, Y$, their product in $\Finset\op$ can be realized as the disjoint union $X+Y$ and the left projection is then realized as the left inclusion $i_X \colon X \to X + Y$. The homomorphism $D_\theory i_X \colon D_\theory X \to D_\theory (X+Y)$ includes the formulae with free variables in $X$ as a sublattice of the formulae with free variables in $X+Y$. Recall from Remark~\ref{rem:descr_of_left_adjoint_syntactic} that its left adjoint $\exists_{i_X} \colon D_\theory (X+Y) \to D_\theory X$ sends a formula $\phi(x,y)$ to $\exists y.\, \phi(x,y)$, where `$x$' denotes the tuple of variables in $X$ and `$y$' denotes the tuple of variables in $Y$. 
Let $t$ be a filter element of $D_\theory^\delta(X+Y)$, and write $\tau := \{\phi \in D_\theory(X+Y) \ | \ t \leq \phi\}$. The $\sigma$-extension map $\exists_{i_X}^\sigma$ sends $t$ to $\bigwedge \{ \exists y . \, \phi \ | \ \phi \in \tau\}$, a filter element in $D_\theory^\delta X$. Since $\class{-}^\sharp_X$ preserves meets, we thus obtain the following expression for the left-hand-side of (\ref{eq:extensible-inequality}):
\begin{equation} \label{eq:consistent-types}
    \class{\exists_{i_X}^\sigma (t)}^\sharp_X = \left\{ m \in M^X \ \middle\vert \ \text{ for all } \phi \in \tau, \mc{M} \models \exists y.\, \phi(m,y) \right\} \ .
\end{equation} 
For the right-hand-side of (\ref{eq:extensible-inequality}), first note, using that $\class{-}_{X+Y}^\sharp$ preserves meets,  
\[ \class{t}_{X + Y}^\sharp = %
\left\{ (m,n) \in M^{X+Y} \ \middle\vert \ \text{ for all } \phi \in \tau, \mc{M} \models \phi(m,n)\right\} \ . \] 
The function $Mi_X \colon M^{X+Y} \to M^X$ sends an $(X+Y)$-tuple $(m,n)$ to the $X$-tuple $m$, and $\exists_{Mi_X} \colon \Pow(M^{X+Y}) \to \Pow (M^X)$ takes the direct image under $Mi_X$. Therefore,
\begin{equation}\label{eq:realized-types}
    \exists_{Mi_X}(\class{t}_{X+Y}^\sharp) = \left\{ m \in M^X \ \middle\vert \ \text{ there exists } n \in M^Y \text{ such that for all } \phi \in \tau, \mc{M} \models \phi(m,n)\right\} \ .
\end{equation}
We will now use the equalities (\ref{eq:consistent-types}) and (\ref{eq:realized-types}) to give a new perspective on the following  notions from positive model theory; cf.~\cite[\S 3.3]{kamsma} for instance.
\begin{definition}\label{df:saturated}
    Let $\mc M$ be a model of a coherent theory $\theory$ and let $X, Y$ be disjoint sets of variables. A \emph{$Y$-type with $X$-parameters in $\mc{M}$} is a pair $(\tau, m)$, where $\tau$ is a set of formulae with free variables in $X \cup Y$, and  $m \in M^X$ is an $X$-tuple of elements of $\mc M$.

    Let $(\tau,m)$ be a $Y$-type with $X$-parameters in $\mc M$. We say that $(\tau,m)$ is:
    \begin{enumerate}
        \item \emph{consistent} if, for every $\phi \in \tau$, there exists $n \in M^Y$ such that $\mc M \models \phi(m,n)$;
        \item \emph{realised} if there exists $n \in M^Y$ such that, for every $\phi \in \tau$, $\mc M \models \phi(m,n)$.
    \end{enumerate}
    The model $\mc M$ is \emph{$\omega$-saturated} if, for any finite $X, Y$, any consistent $Y$-type with $X$-parameters in $\mc M$ is realised.
\end{definition}
\begin{remark}
	We have deliberately omitted the word `positively' from our notion of $\omega$-saturation, since it is clear by context.  See Remark~\ref{rem:p-model2} for further discussion of our choice of terminology.
\end{remark}
\begin{remark}\label{rem:types-primes}
	Note that our notion of type is often called a \emph{partial type} by others (e.g., \cite[Definition 2.2.3]{kamsma}).  Classically, partial types in context $Y$ are distinguished from \emph{complete} types, the maximal (with respect to inclusion) consistent subsets of $D_\theory(Y)$, which are more commonly abbreviated to just `types'.  Note again that our types are with respect to positive logic, and so will in general be different from the types found in classical model theory (cf.\ \cite{haykazyan_types}).
	
	Later, in addition to partial and complete types, we will wish to speak of \emph{prime} types.  In classical logic, i.e., where every formula has a negation, we are motivated to consider complete types since a type $\tau$ is complete if and only if there is a model $\mc{M}$ and some $m \in M^Y$ such that
	\[
	\tau = \lrset{\phi \, \text{classical}}{\mc{M} \vDash \phi(m)}.
	\]
	However, in positive logic, this equivalence breaks down.  Instead, for a filter $\tau$ of coherent formulae with free variables in $Y$, we have the following equivalence (also see Subsection~\ref{subsec:types-doc} below):
	\begin{enumerate}
		\item $\tau$ is of the form $\tp_{\mc{M}}(m) = \lrset{\phi}{\mc{M} \vDash \phi(m)}$ for some model $\mc{M}$ and tuple $m \in M^Y$ (note that this is the \emph{positive} type of $m$), which we call the \emph{type of} $m$;
		\item $\tau$ is \emph{prime} in the sense that $\bot \not \in \tau$ and if $\phi \lor \psi \in \tau$, then either $\phi \in \tau$ or $\psi \in \tau$.
	\end{enumerate}
	For our purposes, it is more fitting to speak of partial types and prime types.  Nonetheless, complete (positive) types play an important role in positive model theory -- they are precisely the types realised by elements of a \emph{positively closed} model (\cite[Proposition 2.2.2]{kamsma}).
\end{remark}
\begin{theorem}\label{thm:saturated}
    Let $\mc M$ be a model of a coherent theory $\theory$, with the associated point ${(M, \class{-}) \colon D_\theory \to \Pow}$. The following are equivalent:
    \begin{enumerate}
        \item the point $(M, \class{-})$ is extensible;
        \item the model $\mc M$ is $\omega$-saturated.
    \end{enumerate}
    In particular, the $\omega$-saturated models of $\theory$ correspond to the complete points of $D_\theory^\delta$.
\end{theorem}
\begin{proof}
By Corollary~\ref{cor:extensible-cjp} instantiated for the morphism $(M,\class{-})$, (i) is equivalent to the statement that $\class{\exists_{i_X}^\sigma (t)}^\sharp_X \subseteq \exists_{Mi_X}(\class{t}_{X+Y}^\sharp)$ holds for any finite sets $X, Y$, with $i_X \colon X \to X+Y$ the left inclusion, and any $t \in \mc F(D_\theory^\delta(X))$. By the equalities (\ref{eq:consistent-types}) and (\ref{eq:realized-types}), this is in turn equivalent to the statement that $\mc M$ is $\omega$-saturated.  The `in particular' statement now follows in light of Remark~\ref{rem:points-as-models}.
\end{proof}
\begin{remark}
	For the reader familiar with the theory of polyadic spaces developed in~\cite{GooMar2024}, we explain how Theorem~\ref{thm:saturated} relates to the definition of `$\omega$-saturated model' given there.

	Let $\theory$ be a single-sorted relational coherent theory\footnote{The definitions in \cite{GooMar2024} are given at a greater level of generality, but for simplicity we specialize to the setting of single-sorted relational coherent theories here.}, with associated syntactic doctrine $D_\theory \colon \Finset \to \DL$.
	In \cite{GooMar2024}, the authors consider the \emph{polyadic Priestley space dual to} $D_\theory$, which is defined as the functor $S_\theory \colon \Finset\op \to \mbf{Priestley}$ obtained by post-composing $D_\theory\op$ with the Priestley duality functor $\DL\op \to \mbf{Priestley}$.
	Thus, for any finite set $n$, the poset underlying $S_\theory(n)$ is $J^\infty D_\theory^\delta(n)$, or alternatively, the prime filters of $D_\theory(n)$ under inclusion. 
	It follows that the canonical extension doctrine $D_\theory^\delta$ is isomorphic to the doctrine $\mathrm{Up} \circ S_\theory$, where $\mathrm{Up} \colon \mbf{Priestley}\op \to \DLplus$ is the functor that sends a Priestley space to the $DL^+$ of up-sets of its underlying poset.

	In this context, a model of $\theory$, i.e., a point of $D_\theory$, yields, via duality, a set $X$ together with a natural transformation from $X^{(-)}$ to $S_\theory$, 
	where $X^{(-)}$ denotes the functor $\Finset\op \to \Set$ that sends a finite set $n$ to $X^n$ and acts on morphisms by precomposition.
	Among all natural transformations $X^{(-)} \Rightarrow S_\theory$, those which correspond to models are precisely those satisfying a certain exactness condition called the weak interpolation property \cite[Proposition~5.6]{GooMar2024}. In this setting, $\omega$-saturated models are then \emph{defined} as those for which a stronger exactness property called the interpolation property holds \cite[Definition~5.7]{GooMar2024}.

	Given a natural transformation $\tau \colon X^{(-)} \Rightarrow S_\theory$, we can take its inverse image component-wise and obtain a natural transformation $\tau^{-1} \colon \mathrm{Up} \circ S_\theory \Rightarrow \mc{P} \circ X^{(-)}$. Under the isomorphism $D_\theory^\delta \cong \mathrm{Up} \circ S_\theory$, $\tau^{-1}$ may be seen as a natural transformation $D_\theory^\delta \Rightarrow \mc{P} \circ X^{(-)}$. Now, $\tau$ has the interpolation property if and only if the pair $(X^{(-)}, \tau^{-1})$ is a $\CohDocplus$-morphism $D_\theory^\delta \to \mc{P}$.
	Thus, $\omega$-saturated models in the sense of \cite{GooMar2024} correspond precisely to $\CohDocplus$-morphisms $D_\theory^\delta \to \mc{P}$. 
	By Theorem~\ref{thm:extensible}, these are precisely the extensible points of $D_\theory$. In conclusion, the $\omega$-saturated models in the sense of \cite{GooMar2024} are the same as ours, by Theorem~\ref{thm:saturated}.
\end{remark}

\begin{example}\label{ex:finitely_many_formulae}
	Suppose that $D_\theory \colon \Finset \to \DL_f \subseteq \DL$ factors through the subcategory of finite distributive lattices.  Since the canonical extension is inert on finite distributive lattices, and so $D_\theory^\delta = D_\theory$, it follows that every point of $D_\theory$ is extensible.  Thus, if $\theory$ is a coherent theory such that, for any $n$, there are a finite number of formulae in $n$ free variables up to $\theory$-equivalence, every model is $\omega$-saturated (cf.\ \cite[Corollary 12.2.13]{hodges}). Note that if $\theory$ is a classical theory, then there are finitely many formulae in $n$ free variables up to equivalence, for any $n$, if and only if $\theory$ is $\omega$-\emph{categorical}, meaning that there is a unique (up to isomorphism) countable model of $\theory$ (see also Example~\ref{ex:omega-categorical}).
\end{example}
\begin{remark}\label{rem:p-model2}
	We feel that to a select group of readers the correspondence between extensible points and $\omega$-saturation may be already known (cf.\ \cite{butz}), but perhaps not as publicised as it ought to be, and so we conclude this section with a few bibliographic remarks.  As discussed in Remark~\ref{rem:p-model1}, the notion of extensible point we have been discussing is called \emph{$p$-model}  in Makkai's terminology~(\cite[\S 1]{makkai_topos_of_types}), which is also the terminology used by Coumans in \cite{coumans_phd,coumans_apal}.  Makkai writes in \cite{makkai_topos_of_types} that the notion of $p$-model is a generalisation of the notion of (classical) $\omega$-saturation, a sentiment that is echoed in \cite[\S 6.3]{garner}, and in \cite[\S 1.4]{makkai_topos_of_types} it is proved that a classically $\omega$-saturated model is a $p$-model (i.e., a positively $\omega$-saturated model in the sense we have used here).  The correspondence is more explicit with Butz's work \cite{butz}, for instance in Proposition 7.10 {\it op.\ cit.}, which when combined with the observations from \cite{coumans_apal} expresses that $D_\theory^\delta$ is (positively) $\omega$-saturated in an internal sense.
\end{remark}

\section{Presheaves of points}\label{sec:presheaf_pts}
Having studied individual models of a coherent theory $\theory$, it is natural to next study categories of models, i.e., subcategories $\Kcat \hookrightarrow \Mod(\theory) \cong \Pt(D_\theory)$.  In fact, it will be more convenient to consider arbitrary functors $\Kcat \to \Pt(D_\theory)$.  This is because, as we shall see in Construction~\ref{dfn:psh-of-pts}, functors $\Kcat \to \Pt(D_\theory)$ bijectively correspond to presheaf-valued points of $D_\theory$, i.e., morphisms of coherent doctrines $D_\theory \to \Sub_{\Set^\Kcat}$, where $\Sub_{\Set^\Kcat}$ is the doctrine of \emph{subpresheaves}.  This correspondence can be summarised in the slogan `models internal to $\Set^\Kcat$ are $\Kcat$-indexed models'.

Inspired by \cite[\S 6]{coumans_apal}, we wish to identify conditions on a class of models $\Kcat$ so that $\Kcat$ fully describes `all the types' of the theory $\theory$, in a way that will be made precise in Theorem~\ref{thm:iso_iff_types_and_homo} below.
\subsection{Presheaf-valued points and presheaves of points}\label{subsec:twoways}
Let $\Kcat$ be a small category. We denote by $\Set^\Kcat$ the category of (covariant) presheaves on $\Kcat$, i.e., the objects are functors $\Kcat \to \Set$, and the morphisms are natural transformations between such functors.
First, we give the necessary details to justify our slogan: `models internal to $\Set^\Kcat$ are $\Kcat$-indexed models'.  In Definition~\ref{dfn:points}, we defined a coherent point of a doctrine $D$ to be a coherent doctrine morphism from $D$ to $\mc{P}$. Generalizing this definition, a \emph{presheaf-valued} coherent point of $D$ is, by definition, a coherent doctrine morphism from $D$ to $\Sub_{\Set^\Kcat}$. Here, note that, since $\Set^\Kcat$ is a coherent category (see, e.g., \cite[p.~34]{elephant}), $\Sub_{\Set^\Kcat}$ is a coherent doctrine by Example~\ref{exa:subdoc}; we return to the structure of the doctrine $\Sub_{\Set^\Kcat}$ in more detail in Subsection~\ref{subsec:subpsh}, where we will show in particular that $\Sub_{\Set^\Kcat}$ is a $DL^+$-doctrine.

The main interest of presheaf-valued coherent points of $D$ is that they correspond to \emph{presheaves of points}. Below, we explicitly show how to obtain a presheaf of points out of a presheaf-valued point.  
In essence, this is a currying argument. 
We will first give the explicit construction, and then explain (cf. Remark~\ref{rem:justify-psh-of-points}) how one might arrive at it in a more principled way using the adjunction between coherent doctrines and coherent categories.
\begin{construction}\label{dfn:psh-of-pts}
	Let $D \colon \cat\op \to \DL$ be a coherent doctrine, and let $\Kcat$ be a small category. Let $(M,\alpha) \colon D \to \Sub_{\Set^\Kcat}$ be a coherent doctrine morphism.
	
	The \emph{associated presheaf of points} is the following functor $\mc{M} \colon \Kcat \to \Pt(D)$. For $k$ an object of $\Kcat$, we define a point $\mc{M}k := (M_k,\alpha_k) \colon D \to \mc{P}$, as follows:
	\begin{itemize}
		\item the functor $M_k \colon \cat\op \to \Set$ is defined by sending $c \in \ob \cat$ to the set $(Mc)k$, and a morphism $f \colon c \to c'$ of $\cat$ to the function $(Mf)_k \colon (Mc)k \to (Mc')k$;
		\item the natural transformation $\alpha_k \colon D \To \mc{P} \circ (M_k)\op$ has its component at $c \in \ob \cat$ defined to be the function $\alpha_{k,c} \colon Dc \to \mc{P}(M_k c)$ that sends $\phi \in Dc$ to 
		\[ \alpha_{k,c}(\phi) := (\alpha_c \phi)k \ , \] 
		i.e., $\alpha_{k,c}(\phi)$ is the subset of $(Mc)k$ obtained by evaluating at the object $k$ the subpresheaf $\alpha_c \phi$ of $Mc$.
	\end{itemize} 
	For $r \colon k \to k'$ a morphism of $\Kcat$, we define $\mc{M}r \colon \mc{M}k \to \mc{M}k'$ to be the natural transformation from $M_k$ to $M_{k'}$ whose component at $c \in \ob \cat$ is the function $(Mc)r \colon (Mc)k \to (Mc)k'$.
\end{construction}
We omit the proof that $\mc{M}$ is indeed a well-defined functor $\mc{M} \colon \Kcat \to \Pt(D)$, and that the data of a morphism $(M,\alpha)$ can be entirely recovered from its associated presheaf of points, so that we in fact have a bijective correspondence between presheaves of points of $D$ and presheaf-valued points of $D$. The following remark justifies this omission.
\begin{remark}\label{rem:justify-psh-of-points}
	We explain how the theory of coherent categories can be used to clarify and justify Construction~\ref{dfn:psh-of-pts}. To do so, and only for the purposes of this remark, we will make use of the 2-category $\mbf{Coh}$ of coherent categories, coherent functors, and natural transformations, and the fact that the functor $\Sub_{(-)} \colon \mbf{Coh} \to \CohDoc$ has a fully faithful left pseudo-adjoint $\mc{A} \colon \CohDoc \to \mbf{Coh}$, cf., e.g., \cite[\S~1]{Pit1983} or \cite[\S~5.1.3]{coumans_phd}.  In essence, the category $\mc{A}(D)$ is the `syntactic category' of functional predicates in the internal language of $D$.
	
	Let $D \colon \mc{C}\op \to \DL$ be a coherent doctrine, let $\mc{K}$ be a small category, and write $A := \mc{A}(D)$. By the pseudo-adjunction $\mc{A} \dashv \Sub_{(-)}$, we have an equivalence of categories,
	\begin{equation}\label{eq:cohdoc-to-doc} 
		\CohDoc(D, \Sub_{\Set^\mc{K}}) \simeq \mbf{Coh}(A, \Set^{\mc{K}}) \ .
	\end{equation}

	The category $\mbf{Coh}(A, \Set^{\mc K})$ is a subcategory of the functor category $[A, \Set^\Kcat]$.
	The universal property of the functor category construction $\Set^{(-)}$ gives an isomorphism of categories
\[
[A, \Set^\Kcat] \cong  [\Kcat \times A,\Set] \cong [\Kcat,\Set^A] \ .
\]
Since finite limits, existential quantification, and finite joins are all computed point-wise in the coherent category $\Set^\Kcat$, the above isomorphism restricts to an isomorphism 
\begin{equation}\label{eq:coh-fun} 
	\mbf{Coh}(A, \Set^\Kcat) \cong [\Kcat, \mbf{Coh}(A, \Set)] \ . 
\end{equation}
Applying the pseudo-adjunction $\mc{A} \dashv \Sub_{(-)}$ once more, and noting that $\Sub_{\Set} = \mc{P}$, we have
\begin{equation} \label{eq:coh-pt}
	\mbf{Coh}(A,\Set) \simeq \CohDoc(D, \mc{P}) = \Pt(D) \ . 
\end{equation}
Putting (\ref{eq:cohdoc-to-doc}), (\ref{eq:coh-fun}), (\ref{eq:coh-pt}) together, we obtain an equivalence 
\[ \CohDoc(D, \Sub_{\Set^\Kcat}) \simeq [\Kcat, \Pt(D)] \ ,\] 
and the object part of this equivalence functor was made explicit in Construction~\ref{dfn:psh-of-pts}.

We will give another viewpoint on morphisms $D \to \Sub_{\Set^\Kcat}$ in Section~\ref{sec:topos}, where we will see that they also correspond to geometric morphisms from the presheaf topos $\Set^\Kcat$ to a topos $\topos^\coh_D$ defined there.
\end{remark}
\begin{remark}
	Recall from Section~\ref{sec:points} that, if $D$ is the syntactic doctrine of a coherent theory $\mbb{T}$, then coherent points of $D$ can be seen as models of $\mbb{T}$. In this case, if $(M,\alpha) \colon D \to \Sub_{\Set^\Kcat}$ is a coherent morphism, then Construction~\ref{dfn:psh-of-pts} yields a functor from $\Kcat$ to models of $\mbb{T}$.
\end{remark}
\subsection{Subpresheaves and orbits of elements}\label{subsec:subpsh}
We now examine the coherent doctrine $\Sub_{\Set^\Kcat}$ in more detail.  
Recall that subobjects in $\Set^\Kcat$ can be described explicitly as follows: for any presheaf $X \colon \Kcat \to \Set$, a \emph{subpresheaf} of $X$ is a functor $Y \colon \Kcat \to \Set$ such that $Yk \subseteq Xk$ for every object $k$ of $\Kcat$, and, for any morphism $f \colon k \to k'$ of $\Kcat$, the function $Yf$ is the restriction of $Xf$ to $Yk$. In other words, a subobject of the presheaf $X$ is uniquely determined by a family of subsets $(Yk)_{k \in \ob \Kcat}$ of $(Xk)_{k \in \ob \Kcat}$ which has the property of being stable under the action of $X$ on morphisms. From this description, it follows that $\Sub_{\Set^\Kcat}(X)$ is a complete lattice, in which joins and meets are computed by taking object-wise unions and intersections, respectively.  It follows that $\Sub_{\Set^\Kcat}(X)$ is a completely distributive lattice.

For any presheaf $X \colon \Kcat \to \Set$, an \emph{element} of $X$ is a pair $(k,x)$, where $k \in \ob \Kcat$ and $x \in Xk$. The \emph{category of elements} of $X$, which we denote by $\el X$, has as its objects the elements of $X$, and as its morphisms $(k,x) \to (k',x')$ those morphisms $f \colon k \to k'$ in $\Kcat$ such that $(Xf)(x) = x'$. The assignment $X \mapsto \el X$ is the object part of an equivalence between $\Set^\Kcat$ and a category of (discrete) \emph{fibrations} over $\Kcat\op$ (see, e.g.,~\cite[Theorem~2.1.2]{LorRie2020}), but we will only use the object part here.%

Note that any element $(k,x)$ of $X$ defines a subpresheaf 
\begin{align*} 
	\orb_{k,x} \colon \Kcat &\to \Set  \\
	k' &\mapsto \{ x' \in Xk' \ | \ \text{ there exists } f \colon k \to k' \text{ such that } (Xf)x = x' \}
\end{align*} 
of $X$, which we call the \emph{orbit} of the element $(k,x)$ in the presheaf $X$. 
\begin{example}
	Taking $\Kcat$ to be a group $G$, viewed as a one-object category, a presheaf is then simply a set $X$ equipped with a $G$-action.  In this setting, our notion of orbit is precisely the usual sense of the word.
\end{example}
\begin{lemma}\label{lem:orbit-cjp}
	For any element $(k,x)$ of $X$, $\orb_{k,x}$ is the smallest subpresheaf $Y$ of $X$ such that $x \in Yk$. 
\end{lemma}
\begin{proof} 
	We have $x \in \orb_{k,x}k$, and if $Y$ is any subpresheaf of $X$ such that $x \in Yk$, then the definition of subpresheaves implies that $\orb_{k,x}k' \subseteq Yk'$ for any object $k'$ of $\Kcat$.
\end{proof}
Lemma~\ref{lem:orbit-cjp} in particular yields that, for any subpresheaf $Y$ of $X$, we have 
\begin{equation}\label{eq:decomp-orbit}
\textstyle	Y = \bigvee_{(k,x) \in \ob \el Y} \orb_{k,x} \ .
\end{equation}
\begin{proposition}\label{prop:sub-psh-dlplus}
For any small category $\Kcat$, and any presheaf $X \colon \Kcat \to \Set$, the completely join-prime elements of $\Sub_{\Set^\Kcat}(X)$ are precisely the orbits of elements of $X$. Moreover, the functor $\Sub_{\Set^\Kcat} \colon \Set^\Kcat \to \DLplus$ is a $DL^+$-doctrine.
\end{proposition}
\begin{proof}
	Let $X \in \Set^\Kcat$. We show that the completely join-prime elements of $\Sub_{\Set^\Kcat}(X)$ are precisely the orbits of elements of $X$.
	If $(Y_i)_{i \in I}$ is a family of subpresheaves of $X$ and $\orb_{k,x} \leq \bigvee_{i \in I} Y_i$, then, since $x \in \orb_{k,x}k$, there exists $i \in I$ such that $x \in Y_i k$, and thus $\orb_{k,x} \leq Y_i$. Conversely, if $J$ is completely join-prime in $\Sub_{\Set^\Kcat}(X)$, then (\ref{eq:decomp-orbit}) implies $J = \orb_{k,x}$ for some element $(k,x)$ of $J$.
	The equality (\ref{eq:decomp-orbit}) shows that $\Sub_{\Set^\Kcat}(X)$ is join-generated by the orbits, which are completely join-prime elements, and hence is a $DL^+$ for any presheaf $X$.
\end{proof}
\begin{remark}\label{rem:orbit-order}
	Let $X \colon \Kcat \to \Set$ be a presheaf. The restriction of the partial order $\leq$ on $\Sub_{\Set^\Kcat}(X)$ to the set $J^\infty(\Sub_{\Set^\Kcat}(X))$ of orbits of elements of $X$ can be described as follows, using Lemma~\ref{lem:orbit-cjp}: for any elements $(k,x)$ and $(k',x')$ of $X$, we have 
	\[ \orb_{k',x'} \leq \orb_{k,x} \iff x' \in \orb_{k,x}k' \iff \text{ there exists } f \colon (k,x) \to (k',x') \text{ in } \el X \ . \]
	Thus, the poset $J^\infty(\Sub_{\Set^\Kcat}(X))$ is the poset reflection of the category $\left(\el X\right)\op$.
\end{remark}

\subsection{Types in a doctrine}\label{subsec:types-doc}
As we already saw in Section~\ref{sec:points} above, the extensibility of a point $(M,\class{-})$ of a doctrine $D_\theory$ is closely related to the types realised by the model $\mc M$ of $\theory$ that $(M, \class{-})$ corresponds to. We now study types in the setting of a general coherent doctrine.
\begin{definition}\label{dfn:type}
Let $D \colon \cat\op \to \DL$ be a coherent doctrine and $(M,\alpha) \colon D \to \mc{P}$ a point of $D$. For $c \in \ob \cat$ and $x \in Mc$, the \emph{type of} the element $x$ in $(M,\alpha)$ is 
\[ \tp_{M,\alpha}(x) :=  \{ \phi \in Dc \ | \ x \in \alpha_{c} \phi \} \ . \] 
\end{definition}
\begin{remark}
Note that $\tp_{M,\alpha}(x)$ is a prime filter of $Dc$.  Although we will not expand upon this point here, it follows from the completeness theorem for coherent logic that every prime filter of $Dc$ is of the form $\tp_{M,\alpha}(x)$ for some point $(M,\alpha) \colon D \to \power$ and $x \in Mc$.
\end{remark}

We can characterise types of elements abstractly in terms of the canonical extension and the lifting $\alpha_c^\sharp$. Indeed, let $\beta_c \colon \mc{P}(Mc) \to D^\delta c$ be the left adjoint of the lifting $\alpha_c^\sharp \colon D^\delta c \to \mc{P}(Mc)$. Then, applying Lemma~\ref{lem:adjoint-type} in the context of Definition~\ref{dfn:type}, we see that, for any $x \in Mc$, $\beta_c(\{x\})$ is the infimum of $\tp_{M,\alpha}(x)$ in $D^\delta c$. We will now generalize this observation to presheaf-valued points.

\begin{notation}
Throughout the rest of this section, let $D \colon \cat\op \to \DL$ be a coherent doctrine, $\Kcat$ a small category, $(M,\alpha) \colon D \to \Sub_{\Set^\Kcat}$ a presheaf-valued point of $D$, and $\mc{M} \colon \Kcat \to \Pt(D)$ the associated presheaf of points in the sense of Construction~\ref{dfn:psh-of-pts}. For any $c \in \ob\cat$, we write $\alpha^\sharp_c \colon D^\delta c \to \Sub_{\Set^\Kcat}(Mc)$ for the lifting of $\alpha_c$, and $\beta_c$ for its left adjoint.
\end{notation}

\begin{lemma}\label{lem:beta-orb}
	 For any $c \in \ob\cat$ and any element $(k,x)$ of $Mc$,
	\begin{equation}\label{eq:beta-orb}
		\beta_c(\orb_{k,x}) = \bigwedge \tp_{M_k,\alpha_k}(x) \ ,
	\end{equation}
	where the infimum is taken in $D^\delta c$. Moreover, $\beta_c(\orb_{k,x}) \in J^\infty (D^\delta c)$.
\end{lemma}
\begin{proof}
	Let $(k,x)$ be an element of $Mc$. By Proposition~\ref{prop:sub-psh-dlplus}, $\orb_{k,x}$ is completely join-prime, and so, applying Lemma~\ref{lem:adjoint-type}, we have 
	\[ \beta_c(\orb_{k,x}) = \bigwedge \{ \phi \in Dc \ | \ \orb_{k,x} \leq \alpha_c \phi \} \ ,\] 
	where the infimum is taken in $D^\delta c$, and this element is completely join-prime in $D^\delta c$.
	By Lemma~\ref{lem:orbit-cjp}, $\orb_{k,x} \leq \alpha_c \phi$ if and only if $x \in (\alpha_c\phi)k$, and the latter is equal to $\alpha_{k,c}(\phi)$ by Construction~\ref{dfn:psh-of-pts}. The desired equality (\ref{eq:beta-orb}) now follows from the definition of $\tp_{M_k,\alpha_k}(x)$.
\end{proof}
\subsection{Plentiful categories of models}
By Proposition~\ref{prop:sub-psh-dlplus}, each presheaf valued point $(M,\alpha) \colon D \to \Sub_{\Set^\Kcat}$ induces a natural transformation $\alpha^\sharp \colon D^\delta \Rightarrow \Sub_{\Set^\Kcat} \circ M\op$.
Using Lemma~\ref{lem:beta-orb}, we can characterise when $\alpha^\sharp$ is pointwise injective and surjective in terms of model-theoretic properties, namely, realisation of types and homogeneity, respectively. In this way, we justify these properties in a purely categorical way. This will also immediately yield a streamlined proof of Theorem~\ref{thm:iso_iff_types_and_homo}, the main result of this section. This result is analogous to \cite[Theorem~40]{coumans_apal} due to Coumans, as we will explain in more depth in Section~\ref{sec:topos}.

\begin{proposition}\label{prop:sharp-mono}
	The following are equivalent:
	\begin{enumerate}
		\item the natural transformation $\alpha^\sharp \colon D^\delta \to \Sub_{\Set^\Kcat} \circ M\op$ is a monomorphism;
		\item for any $c \in \ob \cat$, $\beta_c \colon \Sub_{\Set^\Kcat}(Mc) \to D^\delta c$ is surjective;
		\item for any $c \in \ob \cat$ and $t \in {J}^\infty D^\delta c$, there exist $k \in \ob \Kcat$ and $x \in M_k c$ such that $t = \bigwedge\tp_{M_k,\alpha_k}(x)$.
	\end{enumerate}
\end{proposition}
\begin{proof}
	The equivalence of (i) and (ii) follows from the fact that $\alpha^\sharp$ is monomorphic if and only if it is component-wise injective, if and only if the left adjoints $\beta_c$ are all surjective.

	In view of Lemma~\ref{lem:beta-orb}, property (iii) asserts that, for any $c \in \ob\cat$, the restriction of $\beta_c$ to a function $J^\infty \Sub_{\Set^\Kcat}(Mc) \to J^\infty D^\delta c$ is surjective. Since completely join-prime elements join-generate $D^\delta c$ and $\beta_c$ preserves joins, this yields (iii) $\Rightarrow$ (ii). For (ii) $\Rightarrow$ (iii), if $\beta_c$ is surjective, then in particular $J^\infty D^\delta c$ is contained in the image of $\beta_c$. Then, if $t = \beta_c(Y)$ for some subpresheaf $Y \leq Mc$, writing $Y$ as in (\ref{eq:decomp-orbit}) and using that $\beta_c$ is join-preserving yields some element $(k,x)$ of $Y$ such that $t = \beta_c(\orb_{k,x})$, as required.
\end{proof}
\begin{proposition}\label{prop:sharp-epi}
	The following are equivalent:
	\begin{enumerate}
		\item the natural transformation $\alpha^\sharp \colon D^\delta \to \Sub_{\Set^\Kcat} \circ M\op$ is an epimorphism;
		\item for any $c \in \ob \mc{C}$, $\beta_c \colon \Sub_{\Set^\Kcat}(Mc) \to D^\delta c$ is an order embedding;
		\item for any $c \in \ob \cat, k,k' \in \ob \Kcat$, $x \in M_k c$ and $x' \in M_{k'} c$, if 
		\[ \tp_{M_k,\alpha_k}(x) \subseteq \tp_{M_{k'},\alpha_{k'}}(x') \ , \] 
		then there exists a morphism $r \colon k \to k'$ in $\Kcat$ such that $((Mc)r)(x) = x'$.
	\end{enumerate}
\end{proposition}
\begin{proof}
	In view of Lemma~\ref{lem:beta-orb}, property (iii) asserts that, for any $c \in \ob\cat$, if $\beta_c(\orb_{k',x'}) \leq \beta_c(\orb_{k,x})$, then $x' \in \orb_{k,x}k'$. By Remark~\ref{rem:orbit-order}, the latter is equivalent to $\orb_{k',x'} \leq \orb_{k,x}$. Thus, property (iii) is equivalent to property (ii). The equivalence of (i) and (ii) follows from the fact that $\alpha^\sharp$ is epimorphic if and only if each $\alpha_c^\sharp$ is surjective, if and only if each $\beta_c$ is an order embedding.
\end{proof}
\begin{definition}\label{df:special}
	\begin{enumerate}
		\item 	We say that $(M,\alpha)$ \emph{strictly realises all prime types} if the equivalent conditions of Propostion~\ref{prop:sharp-mono} hold.
		\item We say that $(M,\alpha)$ is \emph{homogeneous} if the equivalent conditions of Proposition~\ref{prop:sharp-epi} hold.
		\item We say that $(M,\alpha)$ is \emph{plentiful} if it is homogeneous and strictly realises all prime types.
	\end{enumerate}
\end{definition}
We remark on the appropriateness of our terminology in Remark~\ref{rem:terminology_realising_types_and_homogeneous}.  
Proposition~\ref{prop:sharp-mono} and~\ref{prop:sharp-epi} together yield the following theorem.
\begin{theorem}[{cf. \cite[Theorem~40]{coumans_apal}}]\label{thm:iso_iff_types_and_homo}
	Let $(M,\alpha) \colon D \to \Sub_{\Set^\Kcat}$ be a presheaf-valued point of a doctrine $D$.  Then $\alpha^\sharp \colon D^\delta \Rightarrow \Sub_{\Set^\Kcat} \circ M\op$ is a natural isomorphism if and only if $(M,\alpha)$ is plentiful.
\end{theorem}
\begin{remark}\label{rem:terminology_realising_types_and_homogeneous}
	\begin{enumerate}
		\item Let $\theory$ be a coherent theory, and let $\tau$ be a prime $X$-type (without parameters). Recall from Definition~\ref{df:saturated} we say that $\tau$ is realised by a tuple $\vec{m} \in \mc{M}^{|X|}$, in a model $\mc{M} \models \theory$, if $\mc{M} \models \phi(\vec{m})$ for each $\phi \in \tau$, or in other words $\tau \subseteq \tp_{\mc{M}}(\vec{m})$.  In Definition~\ref{df:special}, to say that $\vec{m}$ \emph{strictly} realises $\tau$, we instead ask for an \emph{equality} $\tau = \tp_{\mc{M}}(\vec{m})$.
		\item Our terminology of `homogeneous' agrees with (one of) the standard use(s) of the term in model theory when in the following special case.  Let $\theory$ be a classical theory and consider the automorphism group $\Aut(\mc{M})$ of a single model $\mc{M} \vDash \theory$, which can be viewed as a presheaf of points 
		$\Aut(\mc M) \to \Pt(D_\theory)$ 
		in the obvious way.  Recall that $\mc{M}$ is called \emph{homogeneous} if every partial isomorphism $\vec{m} \cong \vec{m}'$ between finite subsets $\vec{m}, \vec{m}' \subseteq \mc{M}$ admits an extension to a total isomorphism of $\mc{M}$, as in the diagram
		\[
		\begin{tikzcd}
			\vec{m} \ar{r}{\sim} \ar[hook]{d} & \vec{m}' \ar[hook]{d} \\
			M \ar[dashed]{r}{\sim} & M.
		\end{tikzcd}
		\]
		There is a partial isomorphism $\vec{m} \cong \vec{m}'$ if and only if $\tp_{\mc{M}}(\vec{m}) = \tp_{\mc{M}}(\vec{m}')$.  Thus, since the types of a classical theory are discretely ordered, $\Aut(\mc{M})$ is homogeneous as in Definition~\ref{df:special} if and only if $\mc{M}$ is homogeneous in the above sense.
	\end{enumerate}
\end{remark}
\begin{remark}
	It follows that if $(M,\alpha) \colon D \to \Sub_{[\Kcat,\Set]}$ is plentiful, then it is also extensible in the sense of Definition~\ref{dfn:extensible} (cf.\ \cite[Lemma 39]{coumans_apal}).  In \cite[\S 6]{coumans_apal}, Coumans's version of the statement of Theorem~\ref{thm:iso_iff_types_and_homo} asked for the further condition that each evaluation $(M_k,\alpha_k)$ is extensible (equivalently, if $D$ is taken as the syntactic doctrine of a coherent theory $\theory$, each evaluation is $\omega$-saturated).  We have observed that this assumption is not necessary; the extensiblity of each $(M_k,\alpha_k)$ follows automatically from the other two conditions.
\end{remark}
\begin{proposition}
	Let $(M,\alpha) \colon D \to \Sub_{\Set^\Kcat}$ be a plentiful presheaf-valued point of $D$.  Then for each object $k \in \Kcat$, the associated point $(M_k, \alpha_k) \colon D \to \power$ (see Construction~\ref{dfn:psh-of-pts}) is extensible.
	
	In particular, if $\Kcat \subseteq \Mod(\theory)$ is a subcategory of models for a coherent theory $\theory$ that strictly realises all prime types and is homogeneous, then each model in $\Kcat$ is $\omega$-saturated.
\end{proposition}
\begin{proof}
	Let $\ev_k \colon \Sub_{\Set^\Kcat} \to \power$ be the morphism of doctrines that evaluates subpresheaves at the object $k \in \Kcat$, i.e., the underlying functor $\Set^\Kcat \to \Set$ of $\ev_k$ sends a presheaf $X$ to the set $Xk$, and a subpresheaf $Y \subseteq X$ to the subset $Yk \subseteq Xk$.  Thus, $\ev_k \circ (M,\alpha) = (M_k, \alpha_k)$.  Since arbitrary meets, joins and images are computed pointwise in $\Set^\Kcat$, it follows that $\ev_k \colon \Sub_{\Set^\Kcat} \to \power$ is a complete morphism of $DL^+$-doctrines.
	
	As $\alpha^\sharp$ is a natural isomorphism, by Remark~\ref{rem:nat-iso-is-doc-morph}, $(M,\alpha^\sharp) \colon D^\delta \to \Sub_{\Set^\Kcat}$ is a morphism of $DL^+$ doctrines.  By taking the composite $\ev_k \circ (M,\alpha^\sharp)$, we obtain our desired extension of $(M_k,\alpha_k)$.
\end{proof}
\begin{example}[Decidable objects]\label{ex:decobj}
	Consider the theory $\dObj$ of \emph{decidable objects}, i.e., the theory with one binary relation symbol $\neq$ and the axioms
	\[
	(x = y) \land (x \neq y) \vdash_{x,y} \bot \quad \text{and} \quad \top \vdash_{x,y} (x=y) \lor (x \neq y).
	\]
	The category of models of this theory is $\Setinj$, the category of sets and injections.  
	We will show in this example that there are (full) subcategories $\Kcat \subseteq \Setinj$ that strictly realise all prime types but are not homogeneous, that are homogeneous but do not strictly realise all prime types, and conclude with a complete characterisation of the plentiful categories of models for $\dObj$.
	
	Let us first describe the prime $X$-types of $\dObj$, i.e., the prime filters of $D_\theory(X)$. 
	For each natural number $n$, we use the shorthand $E_n$ to denote the regular formula expressing that there are at least $n$ elements in the model, 
	\[ E_n := \exists y_1 . \dots . \exists y_n \, \left(\bigwedge_{{k,\ell \leq n}, \, {k \neq \ell}} y_k \neq y_\ell \right) \ .\]
	Using \cite[Lemma D1.3.8]{elephant} and the axioms of $\dObj$, it follows that every coherent formula in a context $X$ is equivalent, modulo the theory $\dObj$, to a disjunction of formulae of the form
	\[
	\left(\bigwedge_{(x_i,x_{i'}) \in I} x_i = x_{i'}\right) \land \left(\bigwedge_{(x_j,x_{j'}) \in J} x_j = x_{j'}\right) \land E_n \ ,
	\]
	for some $n \in \N$ and subsets $I , J \subseteq X^2$ (where $I,J$ can moreover be chosen to be disjoint).
	For any equivalence relation ${\sim}$ on $X$, we write
	\[ S_{\sim} := \left(\bigwedge_{x, y \in X, x \sim y} x = y \right)\land \left(\bigwedge_{x, y \in X, x \not\sim y} x \neq y\right) \ . \]
	The axioms of the theory $\dObj$ imply that a prime $X$-type contains exactly one of $(x = y)$ and $(x \neq y)$, for each $x,y \in X$.
	Thus, a prime type $\tau$ contains the formula $S_{\sim}$ for some equivalence relation ${\sim}$ of the context $X$.   
	By our above analysis of the coherent formulae of $\dObj$, it follows that $\tau$ can be of one of the two species:
	\[
	\tau_{(\sim,n)} = \; \uparrow\! \left(\textstyle S_{\sim} \land E_n \right)  \quad
	\text{or}\quad \tau_{(\sim,\infty)} = {\textstyle\bigcup_{n \in \N}} \uparrow\! \left(S_{\sim} \land E_n \right) \ ,
	\]
	where, for the former species, we require that $n $ is greater than or equal to the number of equivalence classes of $\sim$.  
	That is to say, a prime $X$-type of $\dObj$ expresses which variables in $X$ are equal, which variables are not, and expresses that the underlying cardinality of the model has at least $n$ elements or is infinite. Note that there is an inclusion of types $\tau_{(\sim,a)} \subseteq \tau_{(\sim',b)}$ if and only if the equivalence relations ${\sim}$ and ${\sim'}$ are identical and $a \leq b$, where $a,b  $ denote elements of $ \mathbb{N} \cup \{\infty\}$ with the usual ordering.

	Therefore, for a subcategory of models $\Kcat \subseteq \Setinj$ to \emph{strictly} realise all prime types, $\Kcat$ must contain a finite set of cardinality $n$, for each $n \in \N$, and at least one infinite set.
	
	Now let us consider homogeneity.  Let $Y ,Y'\in \Kcat \subseteq \Setinj$ be sets that are included in the chosen subcategory $\Kcat$, and suppose that the type of some tuple $(y_1 , \dots , y_n) \in  Y^n$ is included in the type of $(y'_1, \dots , y'_n) \in {Y'}^n$, which, as we have just seen, entails that, for any $1 \leq i, j \leq n$,
	\begin{enumerate}
		\item $y_i = y_j$ if, and only if, $y'_i = y'_j$, and
		\item either $Y$ is finite and $|Y'| \geq |Y|$, or both $Y$ and $Y'$ are infinite.
	\end{enumerate}
	If $|Y| \leq |Y'|$, then there exists an injection $Y \hookrightarrow Y'$ that sends $y_i$ to $y_i'$ for each $1 \leq i \leq n$.
	However, if $|Y| > |Y'|$, then there can be no injection $Y \hookrightarrow Y'$.  Thus, as soon as $\Kcat$ contains two infinite sets of differing cardinality, $\Kcat$ is not homogeneous.
	
	Let $\Setinj^{\leq\kappa}$ (respectively, $\Setinj^{< \kappa}$) denote the full subcategory of $\Setinj$ consisting of all sets of cardinality less than or equal to $\kappa$ (resp., strictly less than $\kappa$).  By our above analysis, we conclude the following: the subcategory $\Setinj^{<\omega}$ of finite sets is a homogeneous category of models for $\dObj$, but does not realise all prime types; $\Setinj^{\leq 2^\omega}$ realises all types but is not homogeneous; the category $\Setinj^{\leq \omega}$ of all countable sets is, by contrast, homogeneous and realises all types.  Indeed, a full subcategory $\Kcat \subseteq \Setinj$ is plentiful if and only if $\Kcat$ contains sets of all finite cardinalities and sets of exactly one infinite cardinality.
\end{example}
\begin{example}[Reconstructing $\omega$-categorical theories]\label{ex:omega-categorical}		
	Recall that a classical (single-sorted) theory $\theory$ is said to be $\omega$-categorical if one of the following equivalent conditions is satisfied:
	\begin{enumerate}
		\item\label{enum:omegacat:finite_types} for each natural number $n$, there are finitely many $n$-types;
		\item\label{enum:omegacat:finite_forms} for each $n$, there are finitely many formulae with $n$ free variables up to $\theory$-provable equivalence, or equivalently the associated doctrine $D_\theory$ factors through finite Boolean algebras;
		\item\label{enum:omegacat:isolated} each type is \emph{isolated}, i.e., each ultrafilter of $D_\theory(n)$ is principal;
		\item\label{enum:omegacat:omegacat} the theory $\theory$ has, up to isomorphism, a unique countable model.
	\end{enumerate}
	(Of course, the equivalence of conditions \ref{enum:omegacat:finite_types}, \ref{enum:omegacat:finite_forms} and \ref{enum:omegacat:isolated} is a direct consequence of Stone duality for Boolean algebras, while it is the final condition \ref{enum:omegacat:omegacat} that gives $\omega$-categorical theories their name.)	
	There is a well-established sense in which an $\omega$-\emph{categorical} (classical) theory can be reconstructed from the automorphism group of any countable model \cite[\S 1]{ahlbrandtziegler}, and more generally from any homogeneous model \cite{caramello_galois_theory}.
	
	The salient steps of the assertion that an $\omega$-categorical theory can be reconstructed from the automorphism group of homogeneous model can essentially be deduced from Theorem~\ref{thm:iso_iff_types_and_homo}.  Given a classical, $\omega$-coherent theory $\theory$, since $D_\theory$ factors through finite Boolean algebras, we have that $D_\theory = D_\theory^\delta$, and as $\theory$ is a complete theory with isolated types, any model (strictly) realises all types.  Thus, given a model $\mc{M} \vDash \theory$, %
	viewed as a functor $M \colon \Aut(\mc M) \to \Pt(D_\theory)$, %
	the associated interpretation transformation $\class{-}_M \colon D_\theory \Rightarrow \Sub_{\Aut(\mc{M})} \circ M\op$ is a natural isomorphism if and only if $\mc{M}$ is homogeneous.  In particular, the unique countable model of an $\omega$-categorical theory is homogeneous (\cite[\S 10]{hodges}), and so $\theory$ can be recovered from the automorphism group of its (unique) countable model (cf.\ \cite[\S 1]{ahlbrandtziegler}).  We mention, for interest, that recent work  also explores homogeneity in the context of endomorphism monoids of models \cite{rogers_endomorphism_monoids}.
\end{example}
\section{Recovering the canonical extension for \texorpdfstring{$\omega$}{omega}-stable theories}\label{sec:reconstr}

In this final section, as an application of Theorem~\ref{thm:iso_iff_types_and_homo}, we will show that the canonical extension $D_\theory^\delta$ of the syntactic doctrine for an $\omega$-\emph{stable} theory can be recovered from a subcategory of models of $\theory$ equipped only with knowledge of the underlying sets of these models (but not, for instance, knowledge of the definable subsets, etc.).  To express ``knowledge of the underlying sets'', we employ the following formalism.
\begin{definition}
	Let $\mc C$ and $\mc D$ be categories equipped with functors $A \colon \mc C \to \mc E$ and $B \colon \mc D \to \mc E$.  We say that $\mc C$ and $\mc D$ are \emph{equivalent over} $\mc E$, and write $\mc{C} \simeq_{\mc{E}} \mc{D}$, if there are functors $F : \mc{C} \leftrightarrows \mc{D} : G$ witnessing an equivalence $\mc{C} \simeq \mc{D}$ which moreover commute with the functors $A$ and $B$, as in the diagram:
	\[\begin{tikzcd}[column sep=tiny]
		{\mc{C}} \ar[shift left=1.5]{rr}{F} \ar[bend right]{rd}[']{A}  && {\mc{D}} \ar[shift left=1.5]{ll}{G} \ar[bend left]{ld}{B} \\
		& {\mc{E}}.
	\end{tikzcd}\]
	In other words, $A$ and $B$ are equivalent as objects of the slice category $\mathbf{CAT}/\mc{E}$.
\end{definition}

In our setting, $\mc{C} \subseteq \Str{\Sigma}$ and $\mc{D} \subseteq \Str{\Sigma'}$ will be subcategories of structures, where $\Sigma, \Sigma'$ are a pair of single-sorted signatures, and so are both equipped with forgetful functors $U \colon \mc{C} \to \Set, U' \colon \mc{D} \to \Set$ that send a structure to its underlying set.  Thus, $\mc{C}, \mc{D}$ are equivalent over $\Set$ if, in addition to there being an equivalence $\mc{C} \simeq \mc{D}$, this equivalence preserves the underlying set of each structure.

We are motivated to consider equivalences of categories over a third category since it is precisely the notion of equivalence that arises when considering isomorphisms of doctrines.
\begin{lemma}\label{lem:iso_of_docs_is_equiv_over}
	Let $D, D' \colon \mc{C}\op \to \DL$ be coherent doctrines (respectively, $DL^+$-doctrines) over the same base category.  If $D \cong D'$, then $\Pt(D) \simeq_{\bf{Lex}[\mc{C},\Set]} \Pt(D')$ (resp., $\Ptc(D) \simeq_{\bf{Lex}[\mc{C},\Set]} \Ptc(D')$).
\end{lemma}
\begin{proof}
	The category of points $\Pt(D)$ of $D$ lives over the category $\bf{Lex}[\mc{C},\Set]$ by sending a morphism $(M,\alpha) \colon D \to \power$ to the functor $M \colon \mc{C} \to \Set$.  Using $\beta \colon D \Rightarrow D'$ to denote the natural isomorphism, the equivalence $\Pt(D) \simeq_{\bf{Lex}[\mc{C},\Set]} \Pt(D')$ is witnessed by sending a morphism $(M,\alpha) \in \Pt(D) $ to $(M,\alpha \circ \beta) \in \Pt(D')$, and in reverse sending $(M,\gamma) \in \Pt(D')$ to $(M,\gamma \circ \beta^{-1}) \in \Pt(D)$. 
\end{proof}
\begin{definition}
	Recall that a model $\mc M$ is \emph{$\omega$-saturated} if, for any finite $X,Y$, every consistent $Y$-type with $X$-parameters in $\mc M$ is realised.  We say that $\mc M$ is \emph{countable and saturated} if the underlying set of $\mc M$ is countable and $\mc M$ is $\omega$-saturated.  We use $\omegaSat(\theory)$ to denote the category of countable, saturated models of $\theory$ and model homomorphisms between these.
\end{definition}
\begin{lemma}[cf.\ {\cite[Lemma~10.1.3]{hodges}}]\label{lem:countsat_are_homo}
	Let $\mc{M}$ and $\mc{M}'$ be models of a coherent theory $\theory$ with $\vec{m} \in M^n$ and $\vec{m}' \in {M'}^n$ such that $\tp_{\mc M}(\vec{m}) \subseteq \tp_{\mc M'}(\vec{m}')$.  Suppose also that $\mc{M}$ is countable and $\mc{M}'$ is $\omega$-saturated.  Then there exists a homomorphism of models $f \colon \mc{M} \to \mc{M}'$ sending $\vec{m}$ pointwise to $\vec{m}'$.
\end{lemma}
\begin{proof}
	Since $M$ is countable, fix an enumeration of $M = (m_1, m_2, \dots )$ where the initial segment $(m_1, \dots , m_n)$ is the given tuple $\vec{m} \in M^n$.  We will construct the homomorphism $f \colon \mc{M} \to \mc{M}'$ inductively by building partial structure homomorphisms $f_k \colon \mc{M} \rightharpoondown \mc{M}'$ whose domain at each stage is given by $(m_1, \dots , m_k)$.  As a base case, take $f_n \colon  \mc{M} \rightharpoondown \mc{M}'$ to be the partial structure homomorphism that sends $\vec{m} = (m_1, \dots , m_n)$ pointwise to $\vec{m}' = (m'_1, \dots , m'_n) \subseteq M'$.  This is indeed a partial structure homomorphism since $\tp(\vec{m}) \subseteq \tp(\vec{m}')$.
	
	Suppose we have built $f_k \colon (m_1, \dots , m_k) \to \mc{M}'$ for $k \geq n$.  We obtain $f_{k+1}$ as follows.  Firstly, we set $f_{k+1}(m_i) = f_k(m_i)$ for $i \leq k$.  Turning to $f_{k+1}(m_{k+1})$, let $t(x_1, \dots , x_k, x_{k+1})$ denote the type of $(m_1, \dots , m_{k+1}) \subseteq M^{k+1}$.  The type with (finitely many) parameters $t(f_k(m_1) , \dots , f_k(m_k), x_{k+1})$ is consistent with $\mc{M}'$ since, for each $\phi \in t$, we have that $\exists x_{k+1} \, \phi \in \tp(m_1, \dots , m_k)$, from which we deduce that $\mc{M} \vDash \exists x_{k+1} \,\phi(m_1, \dots , m_k)$ and so $\mc{M}' \vDash \exists x_{k+1} \, \phi(f_k(m_1), \dots, f_k(m_k))$ by the fact that $f_k$ is a structure homomorphism.  Hence, since $\mc{M}'$ is $\omega$-saturated, there is some $m'_{k+1} \in M'$ that realises $t(f_k(m_1) , \dots , f_k(m_k), x_{k+1})$, and we set $f_{k+1}(m_{k+1})$ as $m'_{k+1}$.  The function $f_{k+1} \colon (m_1, \dots , m_{k+1}) \to \mc{M}'$ is a partial structure homomorphism since, by design, $\tp(m_1, \dots, m_{k+1}) \subseteq \tp(f(m_1), \dots , f(m_{k+1}))$. 
	
	The structure morphism $f \colon \mc{M} \to \mc{M}'$ is defined as the union $\bigcup_{k \leq \omega} f_k$, which is total since $(m_1 , m_2, \dots)$ enumerated the whole of $M$.  This is clearly a structure homomorphism, i.e., if $\mc{M}\vDash \phi(\vec{m}'')$ then $\mc{M}' \vDash \phi(f(\vec{m}''))$, since we can always restrict to the partial structure homomorphism $f_k$ for some $k$ large enough that each $m'' \in \vec{m}''$ is contained in $(m_1, \dots , m_k)$.
\end{proof}
\begin{corollary}\label{coro:omegasat_is_homo}
	The category $\omegaSat(\theory)$ of countable, saturated models of $\theory$ is a homogeneous category of models in the sense of Definition~\ref{df:special}.
\end{corollary}
\begin{theorem}\label{thm:if_enough_omegasat_then_reconstr}
	Let $\theory$ and $\theory'$ be coherent theories.  If every prime type of $\theory$ (respectively, $\theory'$) is strictly realised in a countable, saturated model, then the canonical extensions of their associated syntactic doctrines are isomorphic if and only if their countable, saturated models are equivalent over $\Set$, i.e.,
	\[
	D_\theory^\delta \cong D_{\theory'}^\delta \iff \omegaSat(\theory) \simeq_\Set \omegaSat(\theory').
	\]
\end{theorem}
\begin{proof}
	Let us write $\mc K := \omegaSat(\theory)$ and $\mc K' := \omegaSat(\theory')$. Suppose first that $\mc K \simeq_\Set \mc K'$.  Then it is easily seen that there is a natural isomorphism $\Sub_{\Set^{\mc K}} \circ U\op \cong \Sub_{\Set^{\mc K'}} \circ {U'}\op$, where $U$ and $U'$ denote the respective forgetful functors $U \colon \mc K \to \Set$, $U' \colon \mc K' \to \Set$.  By assumption, $\mc K$ strictly realises all prime types, and by Corollary~\ref{coro:omegasat_is_homo} it is also homogeneous (and similarly for $\mc K'$).  Thus, by Theorem~\ref{thm:iso_iff_types_and_homo}, there is a chain of natural isomorphism
	\[
	D_\theory^\delta \cong \Sub_{\Set^{\mc K}} \circ U\op \cong \Sub_{\Set^{\mc K'}} \circ {U'}\op \cong D_{\theory'}^\delta.
	\]
	
	Conversely, suppose that $D_\theory^\delta \cong D_{\theory'}^\delta$. By Lemma~\ref{lem:iso_of_docs_is_equiv_over}, we have 
	\[\Ptc(D_\theory) \simeq_{\bf{Lex}[\Finset,\Set]} \Ptc(D_{\theory'}).\]  Since the equivalence $\Ptc(D_\theory) \simeq \Ptc(D_{\theory'})$ occurs over ${\bf{Lex}[\Finset,\Set]}$, we can restrict this equivalence to $\mc{C} \simeq_{\bf{Lex}[\Finset,\Set]} \mc{C}'$, where $\mc{C} \subseteq \Ptc(D_\theory)$ is the full subcategory of those complete points $(M,\alpha)$ where $M$ takes values in countable sets, and similarly for $\mc{C}' \subseteq \Ptc(D_{\theory'})$.
	Now recall, from Theorem~\ref{thm:saturated}, that an object $(M,\alpha) \in \mc{C}$ is just the datum of a countable, $\omega$-saturated model of $\theory$, and moreover this equivalence commutes with the canonical forgetful functors in the sense that there is a commuting square
	\[
	\begin{tikzcd}
		\Ptc(D_\theory) \ar{r}{\sim} \ar{d} & \omegaSat(\theory) \ar{d} \\
		\bf{Lex}[\Finset,\Set] \ar{r}{\sim} & \Set.
	\end{tikzcd}
	\]
	The same analysis holds for $\theory'$, and so from $\mc{C} \simeq_{\bf{Lex}[\Finset,\Set]} \mc{C}'$ we deduce that $\omegaSat(\theory) \simeq_{\Set} \omegaSat(\theory')$ as desired.
\end{proof}
Thus, we obtain a reconstruction result for the canonical extension of the syntactic doctrine associated with any theory that has enough countable, saturated models.  The existence of countable, saturated models is, in general, a hard problem, but a sufficient condition is found in the model-theoretic notion of $\omega$-\emph{stability}, which imposes a bound on the complexity of models as parameters are introduced.
\begin{definition}\label{df:stable}
	A coherent theory $\theory$ is said to be \emph{$\omega$-stable} if, given any model $\mc{M} \vDash \theory$ and a countable subset $A \subseteq M$, there are countably many prime types with parameters in $A$.
\end{definition}
\begin{proposition}[\cite{hodges}, Theorem 10.2.7]\label{prop:stable_iff_countsat}
	If $\theory$ is $\omega$-stable, every prime type of a coherent theory is strictly realised in a countable, saturated model.
\end{proposition}
While the reader is directed to \cite{hodges} for a detailed proof in the classical setting (from which a version for coherent logic can be derived), we give some intuition about why $\omega$-stable theories have enough countable, saturated models: suppose we started with a countable model $\mc{M}_0$. %
If there are types with finitely many parameters from $\mc{M}_0$ that are not realised in $\mc{M}_0$, we find, using standard model-theoretic techniques, a larger, countable model $\mc{M}_1 \supseteq \mc{M}_0$ that does realises those types.  If $\mc{M}_1$ is not $\omega$-saturated, we repeat to find $\mc{M}_2 \supseteq \mc{M}_1$, and so on.  The bound on complexity of models imposed by $\omega$-stability ensures that after $\omega$-many steps, in the limit $\bigcup_{n < \omega} \mc{M}_n$, the process terminates and we are left with a countable, saturated model as desired.

Thus, by specialising Theorem~\ref{thm:if_enough_omegasat_then_reconstr} we arrive at our reconstruction result for $\omega$-stable theories.
\begin{corollary}\label{cor:omega-stable-reconstruction}
	Let $\theory$ and $\theory'$ be $\omega$-stable coherent theories.  The canonical extensions of their associated syntactic doctrines are isomorphic if and only if their countable, saturated models are equivalent over $\Set$.
\end{corollary}
That is to say, for $\omega$-stable theories, the data of their (prime) types are entirely determined by the underlying sets and functions of their countable, saturated models and homomorphisms thereof.  This hints that the correct notion of \emph{Galois types} (cf., e.g.~\cite[\S 6]{grossberg}), in the case of a coherent theory, ought to involve homomorphisms of models, not only isomorphisms as in the classical case.
\begin{example}
	The assumption that $\theory$ is $\omega$-stable is not necessary in order to have enough countable, saturated models.  For instance, the theory of the Rado (or random) graph gives a counterexample.  To encode this theory in coherent logic, we take three binary relation symbols $\neq, \sim , \not \sim$, and the axioms
	\begin{align*}
		(x=y) \land (x \neq y) & \vdash \bot, & \top & \vdash (x=y) \lor (x \neq y), \\
		(x \sim y) \land (x \not \sim y) & \vdash \bot, & \top & \vdash (x \sim y) \lor (x \not \sim y),
	\end{align*}
	expressing that $=, \neq$ and $\sim, \not \sim$ are complements, $x \sim y \vdash y \sim x$ making $\sim$ symmetric, i.e., the edge relation of a graph, and for all $n, m \in \N$ the axiom
	\[
	\top \vdash \exists y. \left(\bigwedge_{i \leq n+m} y \neq x_i \land \bigwedge_{i \leq n} y \sim x_i \land \bigwedge_{n < j \leq m} y \not \sim x_j \right)
	\]
	expressing that for any two finite subsets of vertices, there is another distinct vertex that is joined to all the vertices in the first set, but none of the latter.  This theory is not $\omega$-stable (\cite[Example 4]{maryanthe_shelah}) but it is an $\omega$-categorical theory (\cite[\S 7.4]{hodges}), and so every model is $\omega$-saturated (via Example~\ref{ex:finitely_many_formulae}).
\end{example}
\section{A tale of two topoi}\label{sec:topos}
Motivating Coumans's work~\cite{coumans_phd,coumans_apal} on the canonical extension of doctrines was the connection to topos theory, which we have so far mostly omitted.  In this concluding section, we describe the topos-theoretic counterparts to the results contained in this paper, which in particular lets us compare our results in more depth to those by Makkai~\cite{makkai_topos_of_types} and Coumans.  We will therefore assume some knowledge of topos theory for this discussion, as can be found in \cite{SGL} or \cite{elephant}.

Given a coherent theory $\theory$, there are two topoi\footnote{Throughout this section, `topos' is used as shorthand for `Grothendieck topos'.} we can naturally associate with the theory $\theory$: its classifying topos and its topos of types.  We begin by briefly recalling these constructions. 

\smallskip
\textit{Classifying topoi.} \quad
The \emph{classifying topos} of a coherent theory $\theory$, which we denote by $\topos_\theory$, is characterised by the property that, for any topos $\ftopos$, geometric morphisms $\ftopos \to \topos_\theory$ correspond to models of $\theory$ internal to $\ftopos$.
Similarly, each coherent $D \colon \cat\op \to \DL$ has a `classifying topos' too. 
Denote by $\topos^\coh_D$ the topos of sheaves on the site consisting of the category $\mc{A}(D)$ equipped with the coherent topology, where $\mc{A}$ is the `syntactic category' functor from Remark~\ref{rem:justify-psh-of-points}. Then the topos $\topos^\coh_D$ \emph{classifies} $D$, in the sense that, for any topos $\ftopos$, we have an equivalence $\Geom(\ftopos,\topos^\coh_D) \simeq \CohDoc(D,\Sub_\ftopos)$.  
In particular, for $\theory$ a coherent theory, the classifying topos $\topos^\coh_{D_\theory}$ of the syntactic doctrine $D_\theory$ is the usual classifying topos of $\theory$.

We have used the superscript $\coh$ to emphasise the fact that $\topos^\coh_D$ classifies the \emph{coherent} morphisms $D \to \Sub_\ftopos$.  If the coherent doctrine $D$ takes values in the subcategory of frames $\Frm \subseteq \DL$, we may also wish to classify morphisms of coherent doctrines $(M,\alpha) \colon D \to \Sub_\ftopos$ where $\alpha \colon D \Rightarrow \Sub_\ftopos \circ M\op$ is a natural transformation of $\Frm$-valued functors. 
We call $D$ a \emph{geometric} doctrine and $(M,\alpha)$ a morphism of \emph{geometric} doctrines (\cite[Definition III.43]{wrigley_phd}).   In this case, there is a topos $\topos^\geom_D$ for which $\Geom(\ftopos,\topos^\geom_D) \simeq \mathbf{GeomDoc}(D,\Sub_\ftopos)$, which can be obtained by taking sheaves on the {geometric} category $\mc{A}(D)$ with the geometric topology.  Note that, in particular, the canonical extension $D^\delta$ of any coherent doctrine is a geometric doctrine.

\smallskip
\textit{Topoi of types.} \quad
In addition to the classifying topos, a coherent theory $\theory$ also has a corresponding \emph{topos of types} $\mc{T}_\theory$, as introduced by Makkai \cite{makkai_topos_of_types}.  
A key contribution of Coumans (\cite[Proposition 24]{coumans_apal}) was to show that %
there is an equivalence
\[\mc{T}_{\theory} \simeq \topos^\geom_{D^\delta_\theory}.\]
Generalising, we write $\mc{T}_D$ for $\topos^\geom_{D^\delta}$ and refer to it as the topos of types for a coherent doctrine $D$.  Thus, for any topos $\ftopos$, geometric morphisms $\ftopos \to \mc{T}_D$ correspond to morphisms of geometric doctrines $D^\delta \to \Sub_\ftopos$.

If $\ftopos$ is a topos such that each subobject lattice of $\ftopos$ is a $DL^+$, i.e., $\Sub_\ftopos$ is a $DL^+$-doctrine, then we can restrict $\Geom(\ftopos,\mc{T}_D) \simeq \Geom(D^\delta,\Sub_\ftopos)$ to an equivalence 
\[\Geom^+(\ftopos,\mc{T}_D) \simeq \CohDocplus(D^\delta,\Sub_\ftopos),\]
where $\Geom^+(\ftopos,\mc{T}_D)$ is defined as the class of geometric morphisms $f \colon \ftopos \to \mc{T}_D$ such that the inverse image $f^\ast$ preserves arbitrary meets of subobjects.  This is the original universal property of $\mc{T}_D$ that Makkai used to characterise the topos of types (\cite[Theorem 1.1]{makkai_topos_of_types}).  Thus, our Theorem~\ref{thm:saturated} is equivalent to the following corollary:
\begin{corollary}
	For $\theory$ a coherent theory, $\Geom^+(\Set,\mc{T}_\theory)$ is equivalent to the category of (set-based) $\omega$-saturated models of $\theory$.
\end{corollary}

Note that in general $\topos^\geom_{D^\delta} \not \simeq \topos^\coh_{D^\delta}$, and also $\Geom^+(\ftopos,\mc{T}_D)$ is strictly contained in $\Geom(\ftopos,\mc{T}_D)$ -- we must specify with respect to which syntax we are classifying $D^\delta$!

\smallskip
\textit{Internal locales.} \quad 
Here is another perspective on the topos of types:  the geometric doctrine $D^\delta \colon \mc{C}\op \to \DLplus$ describes an \emph{internal frame}, or equivalently an \emph{internal locale}, of the presheaf topos $\Set^{\mc{C}\op}$ (see \cite[\S VI.2]{JT}).  As an internal locale, we can take the topos $\Sh_{\Set^{\mc{C}\op}}(D^\delta)$ of \emph{internal sheaves} on $D^\delta$, generalising the usual notion of sheaves on a locale.  The topos $\Sh_{\Set^{\mc{C}\op}}(D^\delta)$ is equivalent to $\mc{T}_D$ (compare the constructions of $\Sh_{\Set^{\mc{C}\op}}(D)$ in \cite[Example~1.3(c)]{moerdijk_contfibr} and $\topos^\geom_{D^\delta}$ in \cite[\S III]{wrigley_phd}).
As a consequence, there is a canonical \emph{localic} geometric morphism $v \colon \mc{T}_D \simeq \Sh_{\Set^{\mc{C}\op}}(D^\delta) \to \Set^{\mc{C}\op}$ (see \cite[Lemma 1.2]{joh_fact}).

In fact, for any coherent doctrine $D \colon \cat\op \to \DL$, there is a corresponding canonical localic geometric morphism $u_D \colon \topos_D^\coh \to \Set^{\mc{C}\op}$ (\cite[\S IV.1]{wrigley_phd}).

\smallskip
\textit{Hyperconnected-localic factorisations.} \quad
Every geometric morphism $f \colon \ftopos \to \mc{G}$ between toposes can be factorised uniquely up to equivalence as
\[
\ftopos \xrightarrow{f_\hyp} \locfact(f) \xrightarrow{f_\loc} \mc{G},
\]
where $f_\hyp$ is \emph{hyperconnected} and $f_\loc$ is \emph{localic} (and thus the factoring topos $\locfact(f)$ is the topos of sheaves on an internal locale of $\mc{G}$; see \cite[\S A.4.6.5]{elephant}). 

Let $D \colon \cat\op \to \DL$ be a coherent doctrine, $\Kcat$ a small category, and let $p \colon \Set^\Kcat \to \topos^\coh_D$ be a geometric morphism. By the characteristic property of $\topos^\coh_D$, $p$ is uniquely induced by some morphism of coherent doctrines $(M,\alpha) \colon D \to \Sub_{\Set^\Kcat}$. 
Consider the hyperconnected-localic factorisation $(p_\hyp,p_\loc)$.  Since the geometric morphism $u_D \colon \topos^\coh_D \to \Set^{\mc{C}\op}$ is localic, the factorisation of the composite $u_D \circ p$ is given by $((u_D \circ p)_\hyp, (u_D \circ p)_\loc) = (p_\hyp, u_D \circ p_\loc)$, as in the following diagram:
\[\begin{tikzcd}
	\Set^\Kcat \ar{rr}{p} \ar{dr}[']{p_\hyp} &                             & \topos^\coh_D \ar{r}{u_D} & \Set^{\cat\op}\ .\\ 
	                                     & \locfact(p) \ar{ur}[']{p_\loc} &                           &
\end{tikzcd}\]

The internal locale of $\Set^{\mc{C}\op}$ corresponding to the localic morphism $u_D \circ p_\loc$ is easy to calculate: it is given by the (geometric) doctrine $\Sub_\Kcat \circ M\op \colon \mc{C}\op \to \Frm$ (see, for instance, \cite[\S 3]{wr_intloc}).  Thus, since there are equivalences $\mc{T}_D \simeq \Sh_{\Set^{\mc{C}\op}}(D^\delta )$ and $\locfact(p) \simeq \Sh_{\Set^{\mc{C}\op}}(\Sub_\Kcat \circ M\op)$, Theorem~\ref{thm:iso_iff_types_and_homo} can also be expressed as follows:
\begin{corollary}
The morphism $(M,\alpha) \colon D \to \Sub_{\Set^\Kcat}$ is plentiful if and only if there is an equivalence of topoi $\mc{T}_D \simeq \locfact(p)$.
\end{corollary}
\noindent
The left-to-right direction of this corollary is how Coumans originally expressed \cite[Theorem 40]{coumans_apal}.

\smallskip
\textit{Comparing internal locales.} \quad
We end by describing the topos-theoretic intuition behind the constructions in Section~\ref{sec:reconstr}.

Let $D, D' \colon \mc{C}\op \to \DL$ be a pair of coherent doctrines over the same base category, and let $\Kcat \subseteq \Pt(D)$ and $\Kcat' \subseteq \Pt(D')$ be subcategories of points, which we recall from Remark~\ref{rem:justify-psh-of-points} can be alternatively viewed as a pair $(M,\alpha) \colon D \to \Sub_{\Set^\Kcat}$ and $(M',\alpha') \colon D' \to \Sub_{\Set^{\Kcat'}}$ of presheaf-valued points, or, as above, as a pair of geometric morphisms $p \colon \Set^\Kcat \to \topos^\coh_D$ and $p' \colon \Set^{\Kcat'} \to \topos^\coh_{D'}$.  Now suppose there is an equivalence $\Kcat \simeq_{\bf{Lex}[\mc{C},\Set]} \Kcat'$ over ${\bf{Lex}[\mc{C},\Set]}$, where the functor $\Kcat \to \bf{Lex}[\mc{C},\Set]$ is the restriction of the canonical forgetful functor $\Pt(D) \to \bf{Lex}[\mc{C},\Set]$, and similarly for $\Kcat'$. It follows that there is an equivalence of topoi $\Set^\Kcat \simeq_{\Set^{\mc{C}\op}} \Set^{\Kcat'}$ over $\Set^{\mc{C}\op}$, i.e., there is an equivalence $\Set^\Kcat \simeq \Set^{\Kcat'}$ that commutes with the geometric morphisms $u_D \circ p$ and $u_{D'} \circ p'$, as in the diagram
\[\begin{tikzcd}[column sep=tiny]
	 \Set^\Kcat \ar{d}[']{p} &[-20pt] \simeq &[-20pt] \Set^{\Kcat'} \ar{d}{p'}\\
	\topos^\coh_D \ar[bend right = 1em]{dr}[']{u_D} & &\topos^\coh_{D'} \ar[bend left = 1em]{dl}{u_{D'}} \\
	& \Set^{\mc{C}\op} .
\end{tikzcd}\]
Hence, by the above analysis, if there is an equivalence $\Kcat \simeq_{\bf{Lex}[\mc{C},\Set]} \Kcat'$, then there is a chain of equivalences of topoi $\locfact(p) \simeq \locfact(u_D \circ p) \simeq  \locfact(u_{D'} \circ p') \simeq \locfact(p')$.  Thus, we obtain the following corollary:
\begin{corollary}
	Suppose that $\Kcat$ and $\Kcat'$ are plentiful for $\theory$, respectively, $\theory'$, and that $\Kcat \simeq_{\bf{Lex}[\mc{C},\Set]} \Kcat'$. Then $\mc{T}_D \simeq_{\Set^{\mc{C}\op}} \mc{T}_{D'}$. 
\end{corollary}
\noindent
\begin{acknowledgement}		
	We thank Mark Kamsma for patiently answering our questions on model theory, no matter how basic.  Both authors acknowledge the financial support from Agence Nationale de la Recherche project \emph{Topology for types and terms} ANR-23-CE48-0012-01.  The second author also acknowledges the support from Marie Sk{\l}odowska-Curie Grant No.\ 101273434.

	\smallskip
	\rightline{\includegraphics[height=1.0cm]{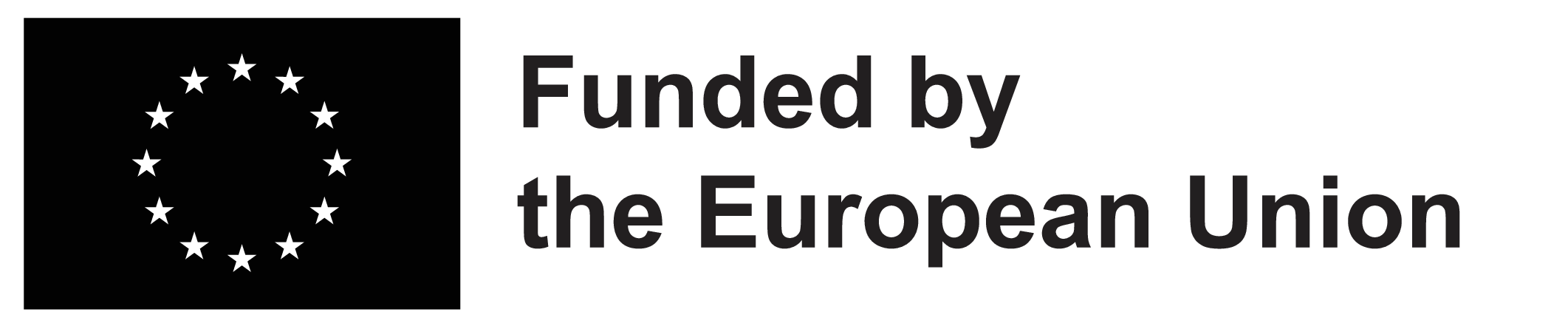}}
\end{acknowledgement}
\bibliographystyle{spmpsci}
\bibliography{biblio.bib}

\begin{thebibliography}{10}
\providecommand{\url}[1]{{#1}}
\providecommand{\urlprefix}{URL }
\expandafter\ifx\csname urlstyle\endcsname\relax
  \providecommand{\doi}[1]{DOI~\discretionary{}{}{}#1}\else
  \providecommand{\doi}{DOI~\discretionary{}{}{}\begingroup \urlstyle{rm}\Url}\fi

\bibitem{ahlbrandtziegler}
Ahlbrandt, G., Ziegler, M.: Quasi finitely axiomatizable totally categorical theories.
\newblock Annals of Pure and Applied Logic \textbf{30}(1), 63--82 (1986)

\bibitem{butz}
Butz, C.: Saturated models of intuitionistic theories.
\newblock Annals of Pure and Applied Logic \textbf{129}(1), 245--275 (2004)

\bibitem{caramello_galois_theory}
Caramello, O.: Topological galois theory.
\newblock Advances in Mathematics \textbf{291}, 646--695 (2016)

\bibitem{coumans_phd}
Coumans, D.: Canonical extensions in logic.
\newblock Ph.D. thesis, Radboud University Nijmegen (2012).
\newblock \urlprefix\url{http://hdl.handle.net/2066/98663}

\bibitem{coumans_apal}
Coumans, D.: Generalising canonical extension to the categorical setting.
\newblock Annals of Pure and Applied Logic \textbf{163}(12), 1940--1961 (2012)

\bibitem{DGP2005}
Dunn, J.M., Gehrke, M., Palmigiano, A.: Canonical extensions and relational completeness of some substructural logics.
\newblock Journal of Symbolic Logic \textbf{70}(3), 713--740 (2005)

\bibitem{FusGeh2025}
Fussner, W., Gehrke, M.: {Canonical Extensions Quickly} (2025).
\newblock \urlprefix\url{https://hal.science/hal-05231128}.
\newblock Preprint

\bibitem{garner}
Garner, R.: Ultrafilters, finite coproducts and locally connected classifying toposes.
\newblock Annals of Pure and Applied Logic \textbf{171}(10), 102831 (2020)

\bibitem{GehHar2001}
Gehrke, M., Harding, J.: Bounded lattice expansions.
\newblock Journal of Algebra \textbf{238}(1), 345--371 (2001)

\bibitem{GehJon1994}
Gehrke, M., Jónsson, B.: Bounded distributive lattice expansions.
\newblock Mathematica Japonica \textbf{40}, 207--15 (1994)

\bibitem{GehJon2004}
Gehrke, M., Jónsson, B.: Bounded distributive lattice expansions.
\newblock Math. Scand. \textbf{94}, 13--45 (2004)

\bibitem{GehPri2007}
Gehrke, M., Priestley, H.: Canonical extensions of double quasioperator algebras: An algebraic perspective on duality for certain algebras with binary operations.
\newblock Journal of Pure and Applied Algebra \textbf{209}(1), 269--290 (2007)

\bibitem{GehPri2008}
Gehrke, M., Priestley, H.A.: Canonical extensions and completions of posets and lattices.
\newblock Reports Math. Log. \textbf{43}, 133--152 (2008)

\bibitem{GehVos2011}
Gehrke, M., Vosmaer, J.: Canonical extensions and canonicity via dcpo presentations.
\newblock Theoretical Computer Science \textbf{412}(25), 2714--2723 (2011)

\bibitem{GooMar2024}
Gool, S.v., Marquès, J.: On duality and model theory for polyadic spaces.
\newblock Annals of Pure and Applied Logic \textbf{175}(2), 103388 (2024)

\bibitem{grossberg}
Grossberg, R.: Classification theory for abstract elementary classes.
\newblock In: Y.~Zhang (ed.) Logic and Algebra, \emph{Contemporary Mathematics}, vol. 302, pp. 165--204. American Mathematical Society (2003)

\bibitem{haykazyan_types}
Haykazyan, L.: Spaces of types in positive model theory.
\newblock The Journal of Symbolic Logic \textbf{84}, 833--848 (2019)

\bibitem{hodges}
Hodges, W.: Model theory.
\newblock Cambridge University Press (1993)

\bibitem{joh_fact}
Johnstone, P.T.: Factorization theorems for geometric morphisms, {I}.
\newblock Cahiers de topologie et g\'eom\'etrie diff\'erentielle \textbf{22}(1), 3--17 (1981)

\bibitem{elephant}
Johnstone, P.T.: Sketches of an Elephant: A Topos Theory Compendium.
\newblock Oxford Logic Guides 43 \& 44. Oxford University Press (2002)

\bibitem{JT}
Joyal, A., Tierney, M.: An extension of the {G}alois theory of {G}rothendieck.
\newblock Memoirs of the American Mathematical Society \textbf{51}(309) (1984)

\bibitem{JonTar1951}
Jónsson, B., Tarski, A.: {Boolean algebras with operators. Part I}.
\newblock American Journal of Mathematics \textbf{73}(4), 891--939 (1951)

\bibitem{JonTar1952}
Jónsson, B., Tarski, A.: {Boolean algebras with operators. Part II}.
\newblock American Journal of Mathematics \textbf{74}(1), 127--162 (1952)

\bibitem{kamsma}
Kamsma, M.: Positive logic: An introduction for model theorists (2025).
\newblock \urlprefix\url{https://arxiv.org/abs/2511.10167}

\bibitem{lawvere}
Lawvere, F.W.: Adjointness in foundations.
\newblock Dialectica \textbf{23}(3/4), 281--296 (1969)

\bibitem{LorRie2020}
Loregian, F., Riehl, E.: Categorical notions of fibration.
\newblock Expositiones Mathematicae \textbf{38}(4), 496--514 (2020)

\bibitem{SGL}
{Mac Lane}, S., Moerdijk, I.: Sheaves in geometry and logic: a first introduction to topos theory.
\newblock Springer (1994)

\bibitem{makkai_topos_of_types}
Makkai, M.: The topos of types.
\newblock In: M.~Lerman, J.H. Schmerl, R.I. Soare (eds.) Logic Year 1979--80, pp. 157--201. Springer Berlin Heidelberg, Berlin, Heidelberg (1981)

\bibitem{makkai-reyes}
Makkai, M., Reyes, G.: First Order Categorical Logic: Model-Theoretical Methods in the Theory of Topoi and Related Categories.
\newblock Lecture Notes in Mathematics. Springer (1977)

\bibitem{maryanthe_shelah}
Malliaris, M., Shelah, S.: General topology meets model theory, on $\mathfrak{p}$ and $\mathfrak{t}$.
\newblock Proceedings of the National Academy of Sciences \textbf{110}(33), 13300--13305 (2013)

\bibitem{moerdijk_contfibr}
Moerdijk, I.: Continuous fibrations and inverse limits of toposes.
\newblock Compositio Mathematica \textbf{58}(1), 45--72 (1986)

\bibitem{MorAlt2018}
Morton, W., van Alten, C.J.: Distributive and completely distributive lattice extensions of ordered sets.
\newblock International Journal of Algebra and Computation \textbf{28}(3), 521--541 (2018)

\bibitem{Pit1983}
Pitts, A.: An application of open maps to categorical logic.
\newblock Journal of Pure and Applied Algebra \textbf{29}(3), 313--326 (1983)

\bibitem{Pri1970}
Priestley, H.A.: {Representation of distributive lattices by means of ordered Stone spaces}.
\newblock {Bull. London Math. Soc.} \textbf{2}, 186--190 (1970)

\bibitem{rogers_endomorphism_monoids}
Rogers, M.: Topological endomorphism monoids of models of geometric theories.
\newblock Theory and Applications of Categories \textbf{42}, 41--58 (2024)

\bibitem{seely}
Seely, R.A.G.: Hyperdoctrines, natural deduction and the {B}eck condition.
\newblock Mathematical Logic Quarterly \textbf{29}(7), 97--173 (1983)

\bibitem{Sto1938}
Stone, M.H.: {Topological representations of distributive lattices and Brouwerian logics}.
\newblock \v{C}asopis pro Pe\v{s}tov\'an\'i Matematiky a Fysiky \textbf{67}, 1--25 (1938)

\bibitem{wrigley_phd}
Wrigley, J.L.: Doctrinal and groupoidal representations of classifying topoi.
\newblock Ph.D. thesis, Università degli Studi dell'Insubria (2024).
\newblock \urlprefix\url{https://jlwrigley.github.io/thesis/thesis_wrigley.pdf}

\bibitem{wr_intloc}
Wrigley, J.L.: Some properties of internal locale morphisms externalised.
\newblock Theory and Applications of Categories \textbf{41}(35), 1160--1202 (2024)

\bibitem{Wri2026}
Wrigley, J.L.: Existential completions and {H}erbrand's theorem.
\newblock Theory and Applications of Categories \textbf{45}(23), 924--950 (2026)

\end{thebibliography}

\end{document}